\documentclass[11pt]{article}
\usepackage{amsmath}
\usepackage{amssymb}
\usepackage{amscd}
\usepackage{amsthm}
\usepackage{amsxtra}
\usepackage{amsbsy}

\newtheorem{theorem}{Theorem}[section]

\newtheorem{proposition}[theorem]{Proposition}
\newtheorem{lemma}[theorem]{Lemma}
\newtheorem{claim}[theorem]{Claim}

\newtheorem{corollary}[theorem]{Corollary}

\newtheorem{D}[theorem]{Definition}
\newenvironment{definition}{\begin{D} \rm }{\end{D}}
\newtheorem{R}[theorem]{Remark}
\newenvironment{remark}{\begin{R}\rm }{\end{R}}

\def\Zee{\mathbb{Z}}
\def\Cee{\mathbb{C}}
\def\Pee{\mathbb{P}}

\def\Pic{\operatorname{Pic}}
\def\Spec{\operatorname{Spec}}
\def\End{\operatorname{End}}
\def\Hom{\operatorname{Hom}}
\def\Ext{\operatorname{Ext}}
\def\Ker{\operatorname{Ker}}
\def\Supp{\operatorname{Supp}}
\def\scrO{\mathcal{O}}
\def\mmu{\boldsymbol{\mu}}

\begin{document}

\title{On the Geometry of Principal Homogeneous Spaces}

\author{A. J. de Jong\thanks{The first author was supported in part by NSF grant \# DMS-0600425.}
\  and Robert Friedman}

\maketitle

\section*{Introduction}

Let $k$ be an algebraically closed field, let $\pi\colon X \to B$ be an elliptic surface defined over $k$, and let $X_K$ be the generic fiber of $\pi$, which is an elliptic curve defined over the field $K=k(B)$, the function field of $B$. If $f\colon Y \to B$ is a genus one fibration locally isomorphic to $X$ (in the \'etale topology on $B$), then $Y$ corresponds to a principal homogeneous space $Y_K$ over $X_K$ which is everywhere locally trivial. The goal of this paper is to study the geometry of such surfaces. Of course, the main  interest is when $X$, $B$ and $Y$ are instead defined over a finite field $\mathbb{F}$. In this case, the set of isomorphism classes of everywhere locally trivial principal homogeneous spaces for $X_K$ is conjectured to be finite. This is known to hold in case  $X \cong D\times B$, where $D$ is an  elliptic curve $D$ defined over $\mathbb{F}$, or $X$ is a rational surface,  by work of Milne \cite{Milne1, Milne2}, or  $X$ is a $K3$ surface defined over $\mathbb{F}$, by work of Artin and Swinnerton-Dyer \cite{ASD}. Since the appearance of \cite{ASD}, little progress has been made in the function field case, and it is our hope that the geometric study of principal homogeneous spaces, over an algebraically closed field, may give some clues as to how to attack the finiteness problem over finite fields. 

 Suppose that $f\colon Y\to B$ is a genus one fibration, everywhere locally trivial. Let $n$ be the smallest positive integer such that there exists an \textsl{$n$-section} of $f$, in other words an irreducible curve $D\subseteq Y$ whose intersection number with a  fiber of $f$ is $n$.  Fix a divisor $D$ on $Y$, not necessarily effective, such that the degree of the restriction $\scrO_Y(D)$ to a smooth fiber is $n$.  For simplicity, we also assume in the introduction  that every fiber of $f$, or equivalently $\pi$, is irreducible. In this case, $D$ is specified by its restriction to a generic fiber up to adding an integral combination of smooth fibers. In particular, the rank $n$ vector bundle $f_*\scrO_Y(D)$ on the base curve $B$ is determined up to tensoring with a line bundle on $B$. Of course, by relative duality, it is essentially equivalent to consider the rank $n$ vector bundle $R^1f_*\scrO_Y(-D)\cong (f_*\scrO_Y(D))\spcheck \otimes \omega^{-1}$. Here $\omega$ is an invertible sheaf on $B$ pulling back to the relative dualizing sheaves $\omega_{Y/\Pee^1}$, resp.\ $\omega_{X/\Pee^1}$ on $Y$, resp.\ $X$. 
If $B\cong \Pee^1$, then we can write
$$R^1f_*\scrO_Y(-D) = \bigoplus _{i=1}^n\scrO_{\Pee^1}(\alpha_i)$$ 
for unique integers $\alpha_1\leq \alpha_2\leq \dots \leq \alpha_n$. Geometrically, the significance of the bundle $R^1f_*\scrO_Y(-D)$ is as follows: for $n\geq 3$, the divisor $D$ is relatively very ample and induces an embedding $Y $ in the projective bundle $\Pee(R^1f_*\scrO_Y(-D))$ over $B$ (and realizes $Y$ as a double cover of $\Pee(R^1f_*\scrO_Y(-D))$ in case $n=2$).

Ideally the vector bundle $R^1f_*\scrO_Y(-D)$ should be fairly well behaved. In particular it is natural to ask if it is semistable. This is not possible for many reasons. For example, if $B\cong \Pee^1$, then a bundle $\bigoplus _{i=1}^n\scrO_{\Pee^1}(\alpha_i)$ is semistable if and only if 
$\alpha_1 = \cdots = \alpha_n$. But $D^2 = -2\sum_i\alpha_i + (n-2)d$, where $d=\chi(Y; \scrO_Y)$, and $D^2 \bmod 2n$, which is clearly an invariant of the restriction of $D$ to the generic fiber,   is an essentially topological invariant. In particular one can show that many values of $D^2\bmod 2n$ occur, so that $R^1f_*\scrO_Y(-D)$ cannot in general be semistable. However, again in the case $B\cong \Pee^1$, one could ask if, in the above notation, $|\alpha_i-\alpha_j| \leq 1$ for all $i,j$. Equivalently, up to a twist, is $R^1f_*\scrO_Y(-D)$ always of the form $\scrO_{\Pee^1}^k\oplus \scrO_{\Pee^1}(-1)^{n-k}$? Define  the pair $(Y,D)$ to be \textsl{rigid} if $B=\Pee^1$ and there exists an integer $t$ such that $f_*\scrO_Y(D)\cong \scrO_{\Pee^1}(t)^k\oplus \scrO_{\Pee^1}(t-1)^{n-k}$.
Thus (for the case $B\cong \Pee^1$), if $(Y,D)$ is a rigid pair, then $R^1f_*\scrO_Y(-D)$ is optimal in various senses: as a bundle over $\Pee^1$, it is both rigid and generic in moduli.

It is easy to see that not every pair $(Y,D)$ is rigid. However,  in some sense, every pair is not too far from being rigid:

\begin{theorem}\label{thm2} Let $f\colon Y \to B\cong \Pee^1$, $D$, and $n$ be as above. Let $d=\chi(X;\scrO_X)$, and suppose that  the characteristic of $k$ does not divide $n$. Let $R^1f_*\scrO_Y(-D) = \bigoplus_{i=1}^n\scrO_{\Pee^1}(\alpha_i)$ 
with $\alpha_1\leq \alpha_2\leq \dots \leq \alpha_n$.  Then 
   $0\leq\alpha_n - \alpha_1 \leq 3d/2$.
\end{theorem} 

In fact, we prove a slightly better bound. In the cases where one can compute all possible examples by hand, namely $n=2,3$ (and presumably $4$), one can check that this bound is essentially sharp.

To put this result in perspective, to obtain further inequalities among the $\alpha_i$, and to generalize to arbitrary base curves, recall that, for a rank $n$ bundle on a curve $B$, the \textsl{slope} $\mu(V)$ of $V$ is the rational number $\deg(V)/n$, and the bundle $V$ is \textsl{semistable} if, for all subbundles $W$ of $V$ with $0<\operatorname{rank}(W)< n$, $\mu(V) \geq \mu(W)$. We then prove the following theorem:

\begin{theorem}\label{thm3} Let $f\colon Y \to B$, $D$, $n$, and $d$ be as above. Suppose that the characteristic of $k$ does not divide $n$. Let $W$ be a subbundle of $R^1f_*\scrO_Y(-D)$ of rank $r$, $0< r<n$, and let $e =\gcd(r+1,n)$. Then
$$\mu(R^1f_*\scrO_Y(-D))-\mu(W) \geq -\left(\frac{r(n-r) + (e-1)}{2nr}\right)d.$$
\end{theorem} 

Essentially trivial manipulations give a corresponding result for quotient bundles:

\begin{corollary}\label{cor4} With the above notation and hypotheses, suppose that $Q$  is a quotient bundle of $R^1f_*\scrO_Y(-D)$ of rank $r$, $0< r<n$, and let $e =\gcd(n-r+1,n)$. Then
$$  \mu(R^1f_*\scrO_Y(-D))-\mu(Q)\leq \left(\frac{r(n-r) + (e-1)}{2nr}\right)d.$$
\end{corollary}

\begin{remark} (i) There are slightly sharper bounds in case $B\cong \Pee^1$. Using these, and considering the appropriate rank one subbundle and quotient bundle of $R^1f_*\scrO_Y(-D)$, one can check that  Theorem~\ref{thm3} and Corollary~\ref{cor4} essentially imply the bounds of Theorem~\ref{thm2}. 

\smallskip
\noindent (ii) One can  drop the hypothesis that the characteristic of $k$ does not divide $n$, but the bounds are not as strong in this case.
\end{remark}

Theorem~\ref{thm3} or the equivalent Corollary~\ref{cor4} say that, while the rank $n$ bundle $R^1f_*\scrO_Y(-D)$ may not be semistable, its failure to be semistable can be controlled in a fairly precise way.

Very roughly speaking, the idea behind the proofs of Theorem~\ref{thm2} and Theorem~\ref{thm3} is as follows. Let $V$ be a vector bundle of rank $r$ on a smooth surface $S$. Define $\Delta(V) = 2rc_2(V) - (r-1)c_1(V)^2$. Then in characteristic  zero, Bogomolov's inequality says that, if $V$ is semistable, then $\Delta(V) \geq 0$. Of course, as is well known, this inequality fails in positive characteristic. However,  vector bundles on elliptic surfaces, and more generally genus one fibrations, which have stable restriction to the generic fiber typically satisfy much stronger forms of Bogomolov's inequality, with only very mild restrictions on the characteristic (see \cite{Fr2, Fr1} for early examples of this in case $V$ has rank $2$). The link between the case where $V$ is semistable and the case where the restriction of $V$ to the generic fiber is stable is as follows: For a genus one fibration $f\colon Y \to B$ and a vector bundle $V$ on $Y$, the semistability of $V$ with respect to an ample divisor numerically equivalent to one of the form $H_0 + tF$, where $F$ is the numerically equivalence class of a general fiber and $t\gg 0$, is closely related to the semistability of the restriction of $V$ to the (geometric) generic fiber of $f$, which is an elliptic curve defined over the algebraic closure of the function field of $B$ \cite{Fr1, Fr2}. For example, let $K=\Spec k(B)$, where $k(B)$ is the function field of $B$, and suppose that the restriction $V_K$ of $V$ to the generic fiber $Y_K$ of $f$, which is a curve of genus one defined over the non-algebraically closed field $K$, is stable. Then it is easy to see that $V$ is semistable with respect to some ample divisor of the form given above.

 We then prove  the following inequality:

\begin{theorem}\label{thm7} Let $f\colon Y \to B$ be a genus one fibration over the algebraically closed field $k$ with $\chi(Y;\scrO_Y)=d$, and let $V$ be a vector bundle of rank $r$ whose restriction to the generic fiber of $f$ is stable. Let $n$ be the degree of the restriction of $V$ to the generic fiber of $f$, and let $e =\gcd(n,r)$.  Suppose that the characteristic of $k$ does not divide $r$. Then
$$\Delta(V) \geq (r^2-e)d.$$
\end{theorem}

(There is also a weaker inequality than that of Theorem~\ref{thm7} which holds with no assumption on  the characteristic of $k$.)

Given Theorem~\ref{thm7}, the proof of Theorem~\ref{thm2} and Theorem~\ref{thm3} becomes a matter of constructing the appropriate vector bundles $V$. There are many approaches to doing so. For example,  Theorem~\ref{thm3} follows by considering certain ``universal extensions." Other methods for constructing bundles lead to results on the vanishing of $H^1(Y; \scrO_Y(D))$ or base point free and embedding theorems for the linear system $|D|$ on $Y$, following methods of Mumford \cite{Mumf} and Reider \cite{Reider}. 

While Theorem~\ref{thm2} and Theorem~\ref{thm3} are valid for all principal homogeneous spaces, one can ask the following question: Fixing a \textbf{generic} elliptic surface $f\colon X\to B$, what can one say about the vector bundles $R^1f_*\scrO_Y(-D)$ for all pairs $(Y,D)$, where  $D$ is an $n$-section. In the case $B\cong \Pee^1$, and $k=\Cee$, a degeneration argument shows:

\begin{theorem}\label{thm9} Let $X$ be a generic elliptic surface over $\Pee^1$. Then  every pair $(Y,D)$ such that $X$ is the Jacobian surface of $Y$ is rigid.
\end{theorem}

It is very likely that the methods of proof can be extended to handle the case of positive characteristic as well.

The contents of this paper are as follows: In Section 1, we establish the Bogomolov type inequality of Theorem~\ref{thm7}. In Section 2, we construct vector bundles for which we can apply this inequality and prove Theorem~\ref{thm3}. In Section 3, we construct examples for small values of $n$ and show that our bounds are sharp. Section 4 studies the group of components of the coarse moduli space of genus one fibrations  over $\Pee^1$. Although we have stated the result in terms of coarse moduli spaces, it could also be formulated via stacks. Unfortunately, in the determination of the monodromy, we eventually have to assume that $k=\Cee$. However, we expect that similar results hold in  the case of positive characteristic as well. Somewhat inconsistently we have devoted a fair amount of space to a proof in all characteristics of a (generalization of) a result of Artin and Swinnerton-Dyer; the proof is straightforward over $\Cee$. Finally, in Section 5, we prove Theorem~\ref{thm9} via a degeneration argument, induction, and a result concerning rational elliptic surfaces.

\medskip

\noindent \textbf{Conventions:} All schemes are separated and of finite type over $k$,  an algebraically closed field. If $X$ is a scheme and $\mathcal{F}$ is a sheaf on $X$, we shall use $H^i(X; \mathcal{F})$ to denote the usual sheaf cohomology in case $\mathcal{F}$ is coherent, \'etale cohomology in case $\mathcal{F}$ is a sheaf of the form $\mmu_n$ or one of its variants, or the Betti cohomology of $X(\Cee)$ in case $X$ is defined over $\Cee$ and $\mathcal{F} =\Zee$. The meaning should be clear from the context. 

\smallskip

A \textsl{genus one fibration $f\colon Y \to B$ without multiple fibers}, or more briefly a \textsl{genus one fibration}, is a smooth projective surface $Y$, together with a morphism $f$ to a smooth projective curve $B$, such that no exceptional curve of $Y$ is contained in a fiber of $f$ ($Y$ is relatively minimal), every fiber of $f$ has trivial dualizing sheaf (hence is of arithmetic genus one), and $f$ has no multiple fibers (if $U$ denotes the  Zariski open subset of $Y$ where $f$ is smooth, then $f(U) = B$). 

\smallskip

Let $f\colon Y\to B$ be a genus one fibration. The \textsl{order} of $Y$ is the smallest positive integer $n$ such that there exists an element $D$ in $\Pic Y$ with $D\cdot F =n$, where $F$ is the class of a general fiber of $f$.

\smallskip

A \textsl{elliptic surface} $\pi : X \to B$ is a genus one fibration endowed
with a section $\sigma : B \to X$.

\section{A Chern class inequality}

 \begin{definition}\label{def1}
Let $V$ be a vector bundle of rank $r$ on a smooth projective surface $S$. Define $\Delta(V) = 2rc_2(V) - (r-1)c_1(V)^2$. By the splitting principle, $\Delta(V) = c_2(End \, V)$. 
\end{definition}

According to Bogomolov's inequality, if $k$ has characteristic zero and $V$ is semistable, then $\Delta(V) \geq 0$. To deal with the analogue of this inequality in positive characteristic, recall that, if $E_K$ is a curve defined over a (not necessarily algebraically closed) field $K$, and $V_K$ is a locally free sheaf on $E_K$, then $V_K$ is \textsl{stable} (resp.\ \textsl{semistable})if, for every subsheaf $W_K$ of $V_K$ (hence defined over $K$), we have $\mu(W_K) < \mu (V_K)$ (resp.\, $\mu(W_K) \leq \mu (V_K)$), where by definition the \textsl{slope} $\mu(W_K)$ of a locally free sheaf on $E_K$ is $\deg (W_K)/\operatorname{rank}W_K$. Standard arguments show:

\begin{enumerate}
\item If $V_K$ is stable and $\varphi\in \End \, V_K$, then either $\varphi=0$ or $\varphi$ is an isomorphism, and hence that $\End \, V_K$ is a  division algebra of finite dimension over $K$.
\item If $V_K$ is stable and $\bar{K}$ is the algebraic closure of $K$, then the pullback $V_{\bar{K}}$ of $V_K$ to $E_K \times _{\Spec K}\bar{K}$ is semistable.
\end{enumerate}

The following theorem is then the first main result of this paper:

\begin{theorem}\label{Cthm1} Let $f\colon Y \to B$ be a genus one fibration  with $\chi(Y;\scrO_Y)=d$, and let $V$ be a vector bundle of rank $r$ whose restriction to the generic fiber of $f$ is stable. Let $n$ be the degree of the restriction of $V$ to the generic fiber of $f$, and let $e =\gcd(n,r)$.  Suppose that the characteristic of $k$ does not divide $e$. Then
$$\Delta(V) \geq (r^2-e)d.$$
If moreover $B\cong \Pee^1$, then
$$\Delta (V) \geq  (r^2-e)d+2(e-1).$$
\end{theorem}
\begin{proof} We begin with the following:

\begin{proposition}\label{prop1} With notation as above, the sheaf $f_*End \, V$ is a locally free sheaf of commutative $\scrO_B$-algebras of rank $e$.
\end{proposition}
\begin{proof}
 Since $f_*End \, V$ is a torsion free 
$\scrO_B$-module, it is locally free. Let $K=k(B)$ be the function field of $B$, and let $Y_K = Y \times _{B}\Spec K$. Then $Y_K$ is a genus one curve over $\Spec K$. Let $V_K$ be the corresponding vector bundle.
It suffices to prove that the restriction of $f_*End \, V$ to $\Spec K$ is a commutative $K$-algebra of dimension $e$. Clearly, this restriction is just $\End \, V_K$.  By hypothesis, $V_K$ is stable and hence $\mathbb{D} = \End \, V_K$ is a  division algebra of finite dimension over $K$. Let $F$ be the center of $\mathbb{D}$. Then $F$ is a finite extension of $K$ and thus is a finitely generated field of transcendence degree one over the algebraically closed field $k$. By Tsen's theorem, the Brauer group of $F$ is trivial. Thus the element of the Brauer group of $F$ defined by  $\mathbb{D}$ is trivial, so that $\mathbb{D} = F$. In particular, $\mathbb{D}$ is commutative. 

Let $\bar{K}$ be the algebraic closure of $K$. To see that the rank of $\mathbb{D}$ is $e$, it suffices to prove the corresponding statement for $\End \, V_{\bar{K}}$, where $V_{\bar{K}}$ is the pullback of $V_K$ to the elliptic curve $Y_{\bar{K}} = Y_K \times _{\Spec K}\Spec \bar{K}$. The bundle $V_{\bar{K}}$ is semistable, and we have  the following well-known lemma:

\begin{lemma}\label{reg} Let $E$ be an elliptic curve defined over an algebraically closed field $\bar{K}$, and let $V_{\bar{K}}$ be a semistable vector bundle on $E$ of rank $r$  and degree $n$.  Let $e =\gcd(n,r)$. Then $\dim _{\bar{K}}\End \, V_{\bar{K}} \geq e$, and $\dim _{\bar{K}}\End \, V_{\bar{K}} = e$ if and only if $\End \, V_{\bar{K}}$ is commutative.
\end{lemma}
\begin{proof} Let $n_0 = n/e$ and $r_0 = r/e$. We  recall some consequences of Atiyah's classification of vector bundles $V_{\bar{K}}$ over an elliptic curve $E=E_{\bar{K}}$ defined over an algebraically closed field (\cite{Atiyah}, Theorem 7 and its corollary): Suppose that $n_0$ and $r_0$ are two relatively prime positive integers and that $\lambda$ is a line bundle on $E$ of degree $n_0$. Then

\begin{enumerate}
\item There is a unique stable bundle $V_{n_0, r_0, 1; \lambda}$ on $E$ of rank $r_0$ and degree $n_0$ such that $\det V_{n_0, r_0, 1; \lambda} =\lambda$.
\item For every  positive integer $d$, there is a unique indecomposable vector bundle $V_{n_0, r_0, d; \lambda}$ on $E$ of rank $dr_0$ and degree $dn_0$, all of whose successive Jordan-H\"older quotients are isomorphic to  $V_{n_0, r_0, 1; \lambda}$.
\item $$\Hom (V_{n_0, r_0, d_1; \lambda_1}, V_{n_0, r_0, d_2; \lambda_2})\cong\begin{cases} 0, &\text{if $\lambda_1\neq \lambda_2$};\\
\bar{K}[t]/(t^k), k=\min(d_1, d_2),&\text{if $\lambda_1= \lambda_2$}.
\end{cases}$$
\item If $V_{\bar{K}}$ is a semistable vector bundle of rank $r$ and degree $n$, with $e =\gcd(n,r)$ and $n/e = n_0$, $r/e=r_0$, then $$V_{\bar{K}}\cong \bigoplus_\lambda V(\lambda),$$
where $V(\lambda)$ is the uniquely defined summand of $V_{\bar{K}}$ which is the (not necessarily direct) sum of all indecomposable summands of $V_{\bar{K}}$ of the form $V_{n_0, r_0, d; \lambda}$ and hence $V(\lambda) \cong \bigoplus _i V_{n_0, r_0, d_i; \lambda}$.
\end{enumerate}

With  $V_{\bar{K}}\cong \bigoplus_\lambda V(\lambda)$ as above, clearly 
$$\End_{\bar{K}} \, V_{\bar{K}}\cong \bigoplus_\lambda\End_{\bar{K}} \, V(\lambda).$$
If  $V(\lambda) \cong \bigoplus _i V_{n_0, r_0, d_i; \lambda}$, 
with $\operatorname{rank} V(\lambda) =  \sum_id_ir_0 = d_\lambda r_0$, then $r= er_0 = \left(\sum_\lambda d_\lambda \right)r_0$, and hence $e= \sum_\lambda d_\lambda$. 
Direct computation then shows that $\End _{\bar{K}}\, V(\lambda)$ is commutative if and only if $V(\lambda) \cong V_{n_0, r_0, d; \lambda}$ is indecomposable, and  also that $\dim \End _{\bar{K}}\, V(\lambda) \geq d_\lambda = \sum_id_i$, with equality holding if and only if $V(\lambda) \cong V_{n_0, r_0, d; \lambda}$ is indecomposable. Combining the two statements gives the proof of the lemma.
\end{proof}

To complete the proof of the proposition, since $\End \, V_{\bar{K}} = \End \, V_K \otimes _{K}\bar{K}$ is commutative, it follows from Lemma~\ref{reg} that  $\dim_{\bar{K}}\End \, V_{\bar{K}} = e$, and hence $\dim_{K}\End \, V_K = e$ as well. Thus $f_*End \, V$ is a locally free sheaf of commutative $\scrO_B$-algebras of rank $e$.
\end{proof}

Let $C = \mathbf{Spec}\, f_*End \, V$ be the scheme over $B$ associated to the coherent sheaf of commutative $\scrO_B$-algebras $f_*End \, V$. Thus there is a finite flat morphism $g\colon C \to B$ of degree $e$. The scheme $C$ is sometimes called the \textsl{spectral cover} of $B$ associated to the bundle $V$. Since the restriction of $f_*End \, V$ to $\Spec K$ is a field, $C$ is an integral scheme of dimension one, and $g_*\scrO_C = f_*End \, V$. Let $\gamma = \deg f_*End \, V =\deg g_*\scrO_C$. 

\begin{lemma}\label{lemma1.5} Suppose that $B$ is a smooth projective curve and that  $g\colon C\to B$ is a finite flat separable morphism of degree $e$, for example suppose that the characteristic of $k$ does not divide $e$. Set $\gamma =\deg g_*\scrO_C$. Then $\gamma \leq 0$. If in addition $B\cong \Pee^1$, then $\gamma \leq -(e-1)$.
\end{lemma}
\begin{proof} Let $\nu\colon \widetilde{C}\to C$ be the normalization. By applying $g_*$ to the exact sequence
$$0 \to \scrO_C \to \nu_*\scrO_{\widetilde{C}} \to Q\to 0,$$
where $Q$ is some torsion sheaf, we see that $\deg g_*\scrO_C \leq \deg (g\circ \nu)_*\scrO_{\widetilde{C}}$, and hence it suffices to prove the result when $C=\widetilde{C}$ is smooth. In this case, a well known formula says that $2 \deg  g_*\scrO_C = -b$, where $b\geq 0$ is the degree of the ramification divisor on the base curve $B$. Thus $\gamma = -b/2\leq 0$. If in addition $g(B) =0$, then the Riemann-Hurwitz formula says that $2g(C)-2 = e(2g(B) -2) + b= -2e+b$. Hence  $\gamma = -b/2 = -(e+g(C) -1) \leq -(e-1)$. 
\end{proof}

We now proceed to calculate the Euler characteristic $\chi(Y; End \, V)$ in two different ways. Since $End \, V$ is a vector bundle of rank $r^2$ on $Y$ with $c_1(End \, V) = 0$ and $c_2(End \, V) = \Delta(V)$, the Riemann-Roch theorem for vector bundles and the fact that $\chi(Y; \scrO_Y) = d$ imply that
$$\chi(Y; End \, V) = -\Delta (V) + r^2d.$$
On the other hand, by the Leray spectral sequence,
$$\chi(Y; End \, V)  = \chi(B; f_*End \, V) - \chi(B; R^1f_*End \, V).$$
Applying Riemann-Roch to the rank $e$ vector bundle $f_*End \, V=g_*\scrO_C$ on $B$  gives
$\chi(B; f_*End \, V) = \gamma + e(1-g)$,
where $g = g(B)$. 
 Applying relative duality to the morphism $f\colon Y\to B$, which is a local complete intersection morphism of relative dimension one, gives
$$(R^1f_*End\, V)\spcheck \cong f_*((End\, V)\spcheck \otimes \omega_{Y/B}).$$
Since $f\colon Y\to B$ is a genus one fibration with no multiple fibers,  $\omega_{Y/B} = f^*\omega$, where $\omega$ is a line bundle on $B$ of degree $d$, and $(End\, V)\spcheck\cong End\, V$. Thus:
$$(R^1f_*End\, V)\spcheck \cong (f_*End\, V) \otimes \omega.$$
This says that $(R^1f_*End\, V)\spcheck{}\spcheck \cong [(f_*End\, V)\spcheck \otimes \omega^{-1}]$, and hence that 
$$R^1f_*End\, V \cong T\oplus [(f_*End\, V)\spcheck \otimes \omega^{-1}],$$
where $T$ is a torsion line bundle on $B$. Now $f_*(End\, V)$ is a vector bundle on $V$ of rank $e$ and degree $\gamma$. Thus $(f_*End\, V)\spcheck$ is a vector bundle on $V$ of rank $e$ and degree $-\gamma$, and $(f_*End\, V) \otimes \omega^{-1}$ has rank $e$ and degree $-\gamma - ed$. Hence 
$$\chi(B; R^1f_*End \, V) = \ell(T) + (-\gamma - ed) + e(1-g).$$
Thus, if $t =\ell(T)$,
\begin{align*}\chi(Y; End \, V)  &= \chi(B; f_*End \, V) - \chi(B; R^1f_*End \, V)\\
&= \gamma + e(1-g) -(t+ (-\gamma - ed) + e(1-g))\\
&= 2\gamma -t + ed \leq 2\gamma  + ed.
\end{align*}
Thus $-\Delta (V) + r^2d \leq 2\gamma  + ed$, and 
after rearranging this gives
$$\Delta (V) \geq r^2d-ed -2\gamma \geq (r^2-e)d.$$
Moreover, if $B\cong \Pee^1$, then 
$$\Delta (V) \geq r^2d-ed -2\gamma \geq (r^2-e)d+2(e-1).$$
This concludes the proof of Theorem~\ref{Cthm1}.
\end{proof}

\begin{remark}  Let $g\colon C \to B$ be the spectral cover of $B$ corresponding to $V$. Clearly $V$ is a module over $f^* f_*End\, V = f^*g_*\scrO_C$.  If $h\colon Y\times _BC \to C$ and $\nu\colon Y\times _BC \to Y$ are the natural morphisms, then by flat base change $f^*g_*\scrO_C = \nu_*h^*\scrO_C = \nu_*\scrO_{Y\times _BC}$. Thus  $V$ corresponds to a sheaf $L$ over $Y\times _BC$, and it is easy to check that $L$ is a  torsion free sheaf on $Y\times _BC$ of rank $r/e$ with $\nu_*L = V$.
\end{remark}

\section{Construction of bundles}

As before, let $f\colon Y \to B$ be a relatively minimal elliptic fibration with generic fiber $F$, let $n$ be the order of $Y$, i.e.\ the smallest positive integer such that there exists a divisor on $Y$ whose intersection number with $F$ is $n$, and let $D$ be a divisor on $Y$ such that $D\cdot F = n$. In case not all fibers of $f$ are irreducible,  we shall make the following assumption on the divisor $D$:

\medskip
\noindent\textbf{Assumption:} If $C$ is a component of a fiber of $f$, then $D\cdot C \geq 0$, i.e.\ $D$ is $f$-nef.
\medskip

For example, this assumption is always satisfied if $D$ is effective and irreducible. If $\ell$ is a divisor on $B$ of degree $t \gg 0$, then $D+f^*\ell$ is effective by Riemann-Roch, and thus linearly equivalent to $D' +\sum _iC_i$, where $D'$ is effective and irreducible by our assumptions on $n$. Hence, for arbitrary $D$, there always exists a $D'$ whose restriction to the generic fiber is linearly equivalent to $D$ and which is $f$-nef. 

A related property of $D$ is the following:

\begin{definition} The divisor $D$ with $D\cdot F =n$ is \textsl{minimal} if $D$ is effective and, for every component $C$ of a fiber of $f$, $D-C$ is not effective.
\end{definition}

Clearly, if $D$ is minimal, then every curve in $|D|$ is reduced irreducible, and hence $D$ is $f$-nef. Given an effective divisor $D$ with $D\cdot F =n$, it is easy to see that there exists a minimal divisor $D_0$ such that $D$ is linearly equivalent to $D_0+ \sum _in_iC_i$, where the $C_i$ are components of fibers and $n_i\geq 0$: Let $H$ be a fixed ample divisor. If $D$ is not minimal, then there exists a component $C_1$ of a fiber such that $D_1 = D-C_1$ is effective. If $D_1$ is minimal, we are done, otherwise we can repeat this process. As $0< H\cdot D_1 < H\cdot D$, we cannot continue indefinitely, so we eventually produce a minimal divisor $D_n=D_0 = D-\sum _in_iC_i$ as desired.

An elementary calculation (for example, using Ramanujam's lemma as in the proof of Proposition~\ref{topsection})  shows that, if $F_0$ is a fiber of $f$ and $\lambda$ is a line bundle on $F_0$ of degree $-n$ such that, for every component $C$ of $F_0$, $\deg (\lambda|C)\leq 0$, then $H^0(F_0; \lambda) =0$ and hence $\dim H^1(F_0; \lambda) = n$. Thus:

\begin{lemma} If $D$ is $f$-nef, then  $R^1f_*\scrO_Y(-D)$ is locally free of rank $n$.
\qed
\end{lemma}

Under our assumption that $D$ is $f$-nef, an easy application of relative duality shows that the rank $n$ bundle $f_*\scrO_Y(D)$ on $B$ is related to $R^1f_*\scrO_Y(-D)$ as follows: 
$$R^1f_*\scrO_Y(-D) \cong [f_*\scrO_Y(D)]\spcheck \otimes \omega^{-1},$$
where as before $\omega$ is the line bundle on $B$ such that $f^*\omega = \omega_{Y/B}$, and hence $\deg \omega = \chi(Y;\scrO_Y) = d$.

\begin{lemma}\label{lemma21} Let $\delta =\deg R^1f_*\scrO_Y(-D)$. Then 
$$D^2 = -2\delta - (n+2)d.$$
Thus, the slope $\mu = \mu (R^1f_*\scrO_Y(-D))$ of $R^1f_*\scrO_Y(-D)$ is
$$\mu =\frac{\delta}{n}=  -\frac{D^2}{2n} -\frac{(n+2)d}{2n}.$$
\end{lemma}
\begin{proof} By Riemann-Roch, $\chi(Y; \scrO_Y(-D)) = \frac12(D^2 + D\cdot K_Y) + \chi(Y; \scrO_Y) = \frac12(D^2 + n(d+2g-2)) + d$. On the other hand,
$$\chi(Y; \scrO_Y(-D)) = \chi(B; R^0f_*\scrO_Y(-D)) -\chi(B; R^1f_*\scrO_Y(-D)) = -(\delta +n(1-g)),$$
using Riemann-Roch on $B$ for the vector bundle $R^1f_*\scrO_Y(-D)$. 
Thus 
$$D^2 + (n+2)d = -2\delta,$$
proving the first statement of the lemma, and the second is then clear.
\end{proof}

\subsection{Universal extensions}

Let $W$ be a subbundle of $R^1f_*\scrO_Y(-D)$ of rank $r$ and degree $\delta_W$. Thus $W$ has slope $\mu(W) = \delta_W/r$. We calculate the Ext group $\Ext^1(f^*W, \scrO_Y(-D))$: 
\begin{align*}\Ext^1(f^*W, \scrO_Y(-D)) &= H^1(Y; f^*W\spcheck \otimes \scrO_Y(-D))\\
= H^0(B; W\spcheck \otimes R^1f_*\scrO_Y(-D)) &= \Hom (W, R^1f_*\scrO_Y(-D)).
\end{align*}

Thus, the inclusion $W\to R^1f_*\scrO_Y(-D)$ defines an extension $V$, given by
$$0 \to  \scrO_Y(-D) \to V \to f^*W \to 0,$$
with the property that the induced coboundary homomorphism $f_*f^*W = W \to R^1f_*\scrO_Y(-D)$ is the given inclusion. One easily calculates the Chern classes of $V$:

\begin{lemma}\label{chcalc} With $V$ as above, the rank of $V$ is $r+1$ and the degree of the restriction of $V$ to the generic fiber of $f$ is $-n$. Moreover,
\begin{align*} c_1(V) &\equiv -D + \delta_W \cdot F;\\
c_2(V) &= -n\delta_W.
\end{align*}
Hence
\begin{align*} \Delta(V) &= 2(r+1)(-n\delta_W) - r(D^2 - 2n\delta_W) = -2n\delta_W -rD^2\\
&= -2n\delta_W +2r\delta + r(n+2)d. \qed
\end{align*}
\end{lemma}

\begin{lemma}\label{stab1} The restriction $V_K$ of $V$ to the generic fiber $Y_K$ of $f$ is stable.
\end{lemma}
\begin{proof} There is an exact sequence
$$0 \to \scrO_{Y_K}(-D) \to V_K \to \scrO_{Y_K}^r \to 0,$$
and the induced homomorphism $H^0(Y_K;\scrO_{Y_K}^r ) \to H^1(Y_K;\scrO_{Y_K}(-D))$ is injective.
The slope of $V_K$ is $-n/(r+1)$.
Let $S$ be a subbundle of $V_K$, with $0< \operatorname{rank} S < r+1$. For the purposes of checking the stability of $V_K$, we may assume that $S$ is semistable. If $\deg S \geq 0$, then  the induced homomorphism $S \to \scrO_{Y_K}^r$ is not zero. Thus the image of $S$ in $\scrO_{Y_K}^r$ has degree $\geq 0$, hence has degree $0$ and is a summand of $\scrO_{Y_K}^r$. Thus $H^0(Y_K;S)$ is a nonzero summand of $H^0(Y_K;\scrO_{Y_K}^r )$ which maps to zero in $H^1(Y_K;\scrO_{Y_K}(-D))$. This contradicts the injectivity of the homomorphism $H^0(Y_K;\scrO_{Y_K}^r ) \to H^1(Y_K;\scrO_{Y_K}(-D))$. Thus, $\deg S < 0$. It follows that $\deg S = -kn$ for some positive integer $k$, and hence $S$ has slope $-kn/t$, where $0<t< r+1$. In this case, the slope of $S$ is clearly $< -n/(r+1)$. Thus $V_K$ is stable. 
\end{proof} 

\begin{theorem}\label{mainthm} With assumptions as above, suppose that the characteristic of $k$ does not divide $n$. Let $W$ be a subbundle of $R^1f_*\scrO_Y(-D)$ of rank $r$, $0< r<n$, and let $e =\gcd(r+1,n)$. Set $\mu = \mu(R^1f_*\scrO_Y(-D))$. Then
$$\mu-\mu(W) \geq -\left(\frac{r(n-r) + (e-1)}{2nr}\right)d.$$
If moreover $B\cong \Pee^1$, then
$$\mu-\mu(W) \geq -\left(\frac{r(n-r) + (e-1)}{2nr}\right)d+\frac{(e-1)}{nr}.$$ 
\end{theorem}
\begin{proof} Lemma~\ref{stab1} and the hypothesis that $\operatorname{char} k$ does not divide $n$ imply that the assumptions of Theorem~\ref{Cthm1} hold. Thus $\Delta(V) \geq ((r+1)^2 -e)d$. By Lemma~\ref{chcalc}, we have
$$-2n\delta_W +2r\delta + r(n+2)d \geq ((r+1)^2 -e)d.$$
Rearranging gives 

$$2r\delta - 2n\delta_W \geq (r^2+2r + 1-e -rn -2r)d = (r(r-n) +(1-e))d.$$
Dividing by $2nr$ gives the inequality.
The last statement follows from the conclusions of Theorem~\ref{Cthm1} in case $B\cong \Pee^1$.
\end{proof}

\begin{corollary}\label{mainthmquot} With assumptions as above, suppose that the characteristic of $k$ does not divide $n$. Let $Q$ be a quotient bundle of $R^1f_*\scrO_Y(-D)$ of rank $r$, $0< r<n$, and let $e =\gcd(n-r+1,n)$. Set $\mu = \mu(R^1f_*\scrO_Y(-D))$. Then
$$\mu-\mu(Q)  \leq \left(\frac{r(n-r) + (e-1)}{2nr}\right)d.$$
If moreover $B\cong \Pee^1$, then
$$\mu-\mu(Q)  \leq \left(\frac{r(n-r) + (e-1)}{2nr}\right)d-\frac{(e-1)}{nr}.$$ 
\end{corollary}
\begin{proof} Let $W$ be the kernel of the surjection $R^1f_*\scrO_Y(-D) \to Q$, so that $W$ has rank $n-r$. Then
$$(n-r)\mu(W) + r\mu(Q) = n\mu = (n-r)\mu + r\mu.$$
Thus
$$\mu(Q) -\mu = \left(\frac{n-r}{r}\right)\left(\mu -\mu(W)\right).$$
By Theorem~\ref{mainthm}, 
\begin{align*}\mu(Q) -\mu &\geq -\left(\frac{n-r}{r}\right)\left(\frac{r(n-r) + (e-1)}{2n(n-r)}\right)d\\
&= -\left(\frac{r(n-r) + (e-1)}{2nr}\right)d.
\end{align*}
Reversing the signs gives the first conclusion of the corollary, and the second is similar.
\end{proof}

In case $B \cong \Pee^1$, we can make the above inequalities more concrete  as follows:

\begin{corollary}\label{P1ineqs} Suppose that $B\cong \Pee^1$ and that the characteristic of $k$ does not divide $n$. Let 
$$R^1f_*\scrO_Y(-D) = \sum_{i=1}^n\scrO_{\Pee^1}(\alpha_i)$$ with $\alpha_1\leq \alpha_2\leq \dots \leq \alpha_n$ and let $e =\gcd(r+1,n)$. Then:
\begin{align*}
\frac{1}{r}\sum_{i=n-r+1}^n\alpha_i -\frac{1}{n}\sum_{i=1}^n\alpha_i&\leq \left(\frac{r(n-r) + (e-1)}{2nr}\right)d-\frac{(e-1)}{nr}\\
&= \frac{(n-r)d}{2n} +\frac{(e-1)(d-2)}{2nr}.\qed
\end{align*}
\end{corollary}

\begin{remark}\label{remark3.9} In the extreme cases of rank one sub- or quotient bundles, these inequalities read:
$$0\leq n\alpha_n-\sum_{i=1}^n\alpha_i \leq\begin{cases} \displaystyle \left(\frac{n-1}{2}\right)d,  &\text{if $n$ is odd;}\\
\displaystyle \left(\frac{nd}{2}\right) -1, &\text{if $n$ is even.}
\end{cases}$$
Moreover,
$$0\leq \sum_{i=1}^n\alpha_i -n\alpha_1\leq (n-1)(d-1).$$
Adding the two inequalities together gives a bound for $\alpha_n -\alpha_1$ on  the order of $3d/2$.
\end{remark}

\subsection{Extensions via linear series}

\begin{lemma}\label{1}
\begin{enumerate}
\item[\rm(i)] Suppose that $h^1(\scrO_Y(D)) \neq 0$. Then there is a non-split extension
$$0 \to \scrO_Y\to V \to \scrO_Y(D-K_Y) \to 0.$$
Moreover $c_1(V) = D-K_Y$ and $c_2(V) = 0$.
\item[\rm(ii)] Suppose that $x$ is a base point of the linear system $|D|$. Then there exists a rank two vector bundle $V$ and an exact sequence
$$0\to \scrO_Y\to V \to \scrO_Y(D-K_Y)\otimes \mathfrak{m}_x\to 0.$$
In particular $c_1(V) = D-K_Y$ and $c_2(V) =1$.
\end{enumerate}
\end{lemma}
\begin{proof} (i) Obvious since $h^1(-(D-K_Y)) = h^1(-D +K_Y) = h^1(D)$.

(ii) The exact sequence
\begin{gather*}\Ext^1(\scrO_Y(D-K_Y)\otimes\mathfrak{m}_x, \scrO_Y) \xrightarrow{\alpha} H^0(Ext^1(\scrO_Y(D-K_Y)\otimes\mathfrak{m}_x, \scrO_Y)) \to \\
\to H^2(Y;\scrO_Y(-D+K_Y)) \to \Ext^2(\scrO_Y(D-K_Y)\otimes\mathfrak{m}_x, \scrO_Y)\end{gather*}
is Serre dual to 
$$H^0(Y;\scrO_Y(D)\otimes\mathfrak{m}_x) \to H^0(Y;\scrO_Y(D)) \xrightarrow{\beta} H^0(k_x) \to H^1(Y; \scrO_Y(D)\otimes\mathfrak{m}_x).$$
Thus a locally free extension $V$ exists $\iff$ $\alpha$ is nonzero $\iff$ $\alpha$ is surjective $\iff$ $\beta=0$ $\iff$ $x$ is a base point of $|D|$.
\end{proof}

\begin{lemma}\label{2} 
\begin{enumerate}
\item[\rm(i)] With   $V$ as in (i) of Lemma~\ref{1}, if $T$ is a divisor class on $Y$ and there exists a nonzero homomorphism $\scrO_Y(T) \to V$, then $\deg_FT\leq 0$. In particular, the restriction of $V$ to the  generic fiber $Y_K$ is stable.
\item[\rm(ii)] With $x$ and $V$ as in (ii) of Lemma~\ref{1}, if $T$ is a divisor class on $Y$ and there exists a nonzero homomorphism $\scrO_Y(T) \to V$, then $\deg_FT\leq 0$. In particular, the restriction of $V$ to the  generic fiber $Y_K$ is stable.
\end{enumerate}
\end{lemma}
\begin{proof} We shall just prove (ii), as the proof of (i) is simpler.  Given $\scrO_Y(T) \to V$, we may assume that the cokernel $V/\scrO_Y(T)$ is torsion free. If the image of $\scrO_Y(T)$ is contained in the subsheaf $\scrO_Y$, then $T = -E$ for some effective divisor $E$. In this case $\deg_FT=-\deg_FE\leq 0$. Otherwise the induced homomorphism $\scrO_Y(T) \to \scrO_Y(D-K_Y)\otimes \mathfrak{m}_x$ is nonzero. Thus $T= D-K_Y-E$ for some effective divisor $E$ with $x\in \Supp E$. If $\deg_FT > 0$, then since $\deg_FT = \deg_FD - \deg_FE \leq n$, and the minimality of $n$, we must have $\deg_FE=0$, and, since $E$ is effective,   $E=\sum_iC_i$ is a sum of irreducible curves $C_i$ contained in fibers of $f$.  On the other hand, since $V/\scrO_Y(T)$ is torsion free, there is an exact sequence
$$0\to \scrO_Y(T) \to V \to \scrO_Y(D-K_Y -T)\otimes I_Z\to 0,$$
where $Z$ is a zero-dimensional subscheme of $Y$. Plugging in $T= D-K_Y- E$, we see that there is an exact sequence
$$0 \to \scrO_Y(D-K_Y -E) \to V \to \scrO_Y(E)\otimes I_Z \to 0.$$
Hence $c_2(V) = (D-K_Y -E)\cdot E +\ell(Z)$.  Since $E$ is supported in the fibers of $f$, $E\cdot K_Y =0$, $E^2 \leq 0$ and $E^2 = 0$ $\iff$ $E$ is numerically equivalent to a multiple $mF$ for some positive integer $m$. First suppose that $E^2<0$. Then $E^2\leq -2$ since $E\cdot K_Y =0$. By our assumption, $D\cdot E\geq 0$, and $K_Y\cdot E =0$, so that $c_2(V) \geq -E^2 \geq 2$, which contradicts $c_2(V)=1$. So $E$ is numerically equivalent to $mF$. In this case  $c_2(V) = (D-K_Y -E)\cdot E +\ell(Z) = m(D\cdot E) +\ell(Z)\geq mn \geq 2$, which again contradicts $c_2(V)=1$. Hence $\deg_FT\leq 0$.
\end{proof}

\begin{remark} One can generalize the above as follows. Let $Z$ be a zero-dimensional subscheme local complete intersection of $Y$; for simplicity assume that $Z=\{p_1, \dots , p_k\}$ consists of $k$ distinct points. Then a locally free extension of the form 
$$0 \to \scrO_Y\to V \to \scrO_Y(D-K_Y)\otimes I_Z \to 0$$
exists $\iff$ $Z$ has the Cayley-Bacharach property with respect to $\scrO_Y(D)$: every section of $\scrO_Y(D)$ which vanishes at all but one of the points $p_i$ vanishes at all of the points.

Assume that a locally free extension $V$ exists, and further assume for simplicity that all fibers of $f$ are irreducible. Let $m=\#\{f(p_1),\dots, f(p_k)\}$. Hence, if $F_1, \dots, F_r$ are distinct fibers of $f$ and if $Z\subseteq F_1 + \cdots + F_r$, then $r\geq m$. Arguments as above show that, if $mn \geq k+1$, then every locally free $V$ as above restricts to a stable bundle on the generic fiber.
\end{remark}

\begin{corollary}\label{2cor4} Suppose that $n=\deg_FD$ is odd.
\begin{enumerate}
\item[\rm(i)] If $h^1(D) \neq 0$, then 
$$D^2 \leq 2n(d+2g-2) -3d =(2n-3)d +4n(g-1).$$
\item[\rm(ii)] If $x$ is a base point of the linear system $|D|$, then 
$$D^2 \leq 2n(d+2g-2) -3d +4=(2n-3)d +4n(g-1)+4.$$
\end{enumerate}
\end{corollary}
\begin{proof} Again, we shall just prove (ii). If such an $x$ exists, then we have constructed a rank two bundle $V$ with $c_1(V) = D-K_Y$ and $c_2(V) =1$ such that the restriction of $V$ to the generic fiber is semistable. The hypotheses of Theorem~\ref{Cthm1} apply, with $e=1$. Thus 
$$4c_2(V) - c_1(V)^2 = 4-D^2 +2n(d+2g-2)\geq 3d,$$
which yields the inequality
$$D^2 \leq 2n(d+2g-2) -3d +4=(2n-3)d +4n(g-1)+4.$$
\end{proof}

Similar arguments show:

\begin{corollary}\label{cor1.6} Suppose that $n=\deg_FD$ is even and that $\operatorname{char} k\neq 2$.
\begin{enumerate}
\item[\rm(i)] If $h^1(D) \neq 0$, then 
$$D^2 \leq  (2n-2)d +4n(g-1).$$
If moreover $B\cong \Pee^1$, then 
$$D^2 \leq(2n-2)d -4n-2.$$
\item[\rm(ii)] If $x$ is a base point of the linear system $|D|$, then 
$$D^2 \leq  (2n-2)d +4n(g-1)+4.$$
If moreover $B\cong \Pee^1$, then 
$$D^2 \leq(2n-2)d -4n+2.\qed$$
\end{enumerate}
\end{corollary}

\section{Examples}

Let $f\colon Y\to B$ be a genus one  fibration and let $D$ be an $f$-nef divisor on $Y$ of relative degree $n>1$. Consider the $\Pee^{n-1}$-bundle $\Pee =\Pee(f_*\scrO_Y(D)\spcheck)=\Pee(R^1f_*\scrO_Y(-D))$ over $B$ (here our conventions are opposite to those of EGA or Hartshorne). Note that replacing $D$ by $D+tF$ leaves the projective space bundle $\Pee$ unchanged. Denote by $\phi\colon \Pee\to B$ the projection and let $P\subseteq \Pee$ be the divisor class of a fiber. Let $\bar{Y}$ be the normal surface obtained by contracting all of the curves which are orthogonal to $D$. There is a morphism $g$ (of  schemes over $B$) from $Y$ to $\Pee$. For $n\geq 3$, $g$ induces an embedding $\bar{g}$  of $\bar{Y}$. For $n=2$, the corresponding morphism $\bar{g}\colon \bar{Y}\to \Pee$ is a double cover. By construction, $g^*\scrO_{\Pee}(1)= \scrO_Y(D)$, and hence  $g^*[\scrO_{\Pee}(1)\otimes \phi^*\lambda] = \scrO_Y(D)\otimes f^*\lambda$. 

\begin{lemma}\label{4.1} For all line bundles $\lambda$ on $B$ and all $i\geq 0$, $g^*$ induces an isomorphism
$$H^i(\Pee; \scrO_{\Pee}(1)\otimes \phi^*\lambda) \cong H^i(Y;\scrO_Y(D)\otimes f^*\lambda).$$
\end{lemma}
\begin{proof} First assume that $Y=\bar{Y}$ and hence $D|E$ is ample for every fiber $E$ of $f$. First suppose that $n\geq 3$, so that $g$ is an embedding. There is an exact sequence
$$0 \to \scrO_{\Pee}(1)\otimes I_Y\to \scrO_{\Pee}(1)\to \scrO_Y(D) \to 0,$$
so it suffices to show that $H^i(\Pee; \scrO_{\Pee}(1)\otimes I_Y\otimes \phi^*\lambda) =0$ for all $i$. Via Leray, it suffices to show that $R^i\phi_*(\scrO_{\Pee}(1)\otimes I_Y) =0$ for all $i$. Since $\scrO_{\Pee}(1)\otimes I_Y$ is flat over $B$, it suffices by base change to show that, for every fiber $P$  of $\phi$, $H^i(P; \scrO_{\Pee}(1)\otimes I_Y \otimes \scrO_P) = 0$. It is easy to check that, if $E = P\cap Y$ is the corresponding fiber of $f$, then $\scrO_{\Pee}(1)\otimes I_Y \otimes \scrO_P = \scrO_{\Pee^{n-1}}(1)\otimes I_E$.  Then there is the corresponding exact sequence
$$0\to \scrO_{\Pee^{n-1}}(1)\otimes I_E\to \scrO_{\Pee^{n-1}}(1)\to \scrO_E(D) \to 0.$$
For $i>0$, $H^i(\Pee^{n-1};\scrO_{\Pee^{n-1}}(1)) = H^i(E;\scrO_E(D))=0$, and for $i=0$ the restriction homomorphism $H^0(\Pee^{n-1};\scrO_{\Pee^{n-1}}(1)) \to H^0(E;\scrO_E(D))$ is an isomorphism. It follows that $H^i(\Pee^{n-1};\scrO_{\Pee^{n-1}}(1)\otimes I_E)=0$ for every $i$. Thus $R^i\phi_*\scrO_{\Pee}(1)\otimes I_Y =0$ for all $i$ and we are done in the case $n\geq 3$.

For $n=2$, $Y$ is a double cover of $\Pee$ branched along a section of $\scrO_{\Pee}(4)\otimes \phi^*\mu^{\otimes 2}$ for some line bundle $\mu$ on $B$. Let $L = \scrO_{\Pee}(2)\otimes \phi^*\mu$. Then $$g_*\scrO_Y(D) = \scrO_{\Pee}(1) \oplus [\scrO_{\Pee}(1)\otimes L^{-1}].$$
So it is enough to show that $H^i(\Pee; \scrO_{\Pee}(1)\otimes L^{-1}\otimes \phi^*\lambda) =0$ for all $i$. Since $\scrO_{\Pee}(1)\otimes L^{-1}$ restricts to $\scrO_{\Pee^1}(-1)$ on every fiber of $\phi$, $R^i\phi _*\scrO_{\Pee}(1)\otimes L^{-1} =0$ for all $i$ and thus $H^i(\Pee; \scrO_{\Pee}(1)\otimes L^{-1}\otimes \phi^*\lambda) =0$  as well.

Finally, in case $Y\neq \bar{Y}$, let $\rho\colon Y \to \bar{Y}$ be the natural morphism. Then $\rho$ is a resolution of singularities. Note that $\bar{Y}$ has rational singularities and $D$ induces a relatively ample Cartier divisor $\bar{D}$ on $\bar{Y}$ with $\rho^*\bar{D} = D$. The above arguments show that $H^i(\Pee; \scrO_{\Pee}(1)\otimes \phi^*\lambda) \cong H^i(\bar{Y};\scrO_{\bar{Y}}(\bar{D})\otimes \bar{f}^*\lambda)$, where $\bar{f}\colon \bar{Y} \to B$ is the induced morphism. Again using the fact that $\bar{Y}$ has rational singularities, $R^0\rho_*(\scrO_Y(D)\otimes f^*\lambda)= \scrO_{\bar{Y}}(\bar{D})\otimes \bar{f}^*\lambda$ and $R^i\rho_*(\scrO_Y(D)\otimes f^*\lambda)=0$ for $i>0$, so we are done via Leray again.
\end{proof} 

For the rest of this section, we assume that $B\cong \Pee^1$ and that $f_*\scrO_Y(D) = \bigoplus _{i=1}^n\scrO_{\Pee^1}(-\alpha_i)$, where $\alpha_1\leq \alpha_2\leq \dots \leq \alpha_n$.  By convention, we choose the line bundle $\scrO_{\Pee}(1)$ such that $\phi_*\scrO_{\Pee}(1) = \bigoplus _{i=1}^n\scrO_{\Pee^1}(-\alpha_i)$. 

\begin{lemma}\label{H1nonzero} In the above notation, $H^1(\Pee; \scrO_{\Pee}(1)\otimes \scrO_{\Pee}(tP)) \neq 0$ $\iff$ $t\leq \alpha_n -2$. Hence, $H^1(Y; \scrO_Y(D+tF)) \neq 0$ $\iff$ $t\leq \alpha_n -2$.
\end{lemma}
\begin{proof} Immediate from the Leray spectral sequence and Lemma~\ref{4.1}.
\end{proof}

Next let us describe linear systems on $\Pee =\Pee(\bigoplus _{i=1}^n\scrO_{\Pee^1}(\alpha_i))$. For $1\leq i\leq n$, the summand $\scrO_{\Pee^1}(\alpha_i)$ gives a section $\Sigma_i$ of $\scrO_{\Pee}(1) \otimes \scrO_{\Pee}(\alpha_iP)$. In particular $\Sigma_1$ is a section of $\scrO_{\Pee}(1) \otimes \scrO_{\Pee}(\alpha_1P)$. Of course, the summand is not unique except in the case $i=1$ and $\alpha_2 > \alpha_1$. Suppose that $\alpha_1 = \alpha_2 =\dots = \alpha_{i_1} < \alpha_{i_1+ 1} = \alpha_{i_1+ 2}=\dots = \alpha_{i_2}< \dots <\alpha_{i_k+ 1}= \dots = \alpha_n$. Consider sections of the line bundle $\scrO_{\Pee}(1) \otimes \scrO_{\Pee}(tP)$. By direct inspection,

\begin{lemma}  In the above notation, if $\alpha_{i_j}\leq t< \alpha_{i_{j+1}}$, then the base locus of the complete linear system corresponding to the line bundle $\scrO_{\Pee}(1) \otimes \scrO_{\Pee}(tP)$ is $\Sigma_1 \cap \Sigma_2\cap \cdots \cap \Sigma_{i_j}$. In particular,  this base locus is empty $\iff$  $t\geq \alpha_{i_k+1}=\alpha_n$. \qed
\end{lemma} 

\begin{corollary}\label{baseloc} Suppose that $n> 2$. The complete linear system $|D+tF|$ has a nonempty base locus $\iff$ $t\leq  \alpha_n-1$ and, for the unique  $j$ such that $\alpha_{i_j}\leq t< \alpha_{i_{j+1}}$, $Y\cap \Sigma_1 \cap \Sigma_2\cap \cdots \cap \Sigma_{i_j} \neq \emptyset$. 

If $n=2$, the complete linear system $|D+tF|$ has a nonempty base locus 
$\iff$ $t\leq  \alpha_2-1$. \qed
\end{corollary}

We now look at examples for small values of $n$.
\medskip

\noindent \textbf{The case $n=2$:} Let $Y$ be a genus one fibration of order $2$. For simplicity, we shall just look at the case $Y =\bar{Y}$, i.e.\ $D$ is $f$-ample. We can normalize $D$ so that the rank two vector bundle $f_*\scrO_Y(D)$ is equal to $\scrO_{\Pee^1}\oplus  \scrO_{\Pee^1}(-a)$, for some $a\geq 0$, and hence $\Pee = \mathbb{F}_a$. There is thus a degree two morphism $g\colon Y \to \mathbb{F}_a$. With this normalization, $D=g^*\sigma_0$,  where $\sigma_0$ is the negative section (or $\sigma_0^2=0$ if $a=0$). Our bounds in the previous section all give the single inequality $a\leq d-1$. If $B\subseteq \mathbb{F}_a$ is the branch divisor of $f$, then $B = 4\sigma_0 + 2NP$ where $P$ is the class of a fiber of the ruled surface. Since $Y$ has no section $B$ is smooth and irreducible with $N\geq 2a$. But $N=2a$ is impossible since then $B\cap \sigma_0 =\emptyset$ and the inverse image of $\sigma_0$ on $Y$ would consist of two sections. Thus $N\geq 2a+1$. Note that $g_*\scrO_Y =  \scrO_{\mathbb{F}_a}\oplus \scrO_{\mathbb{F}_a}(-2\sigma_0 -NP)$. The canonical divisor  $K_Y =g^*(-2\sigma_0 -(a+2)P + 2\sigma_0 + NP) = g^*(N-a-2)P$. Thus $d= N-a$. If $N = 2a +\varepsilon$, then $d = a+\varepsilon$ with $\varepsilon \geq 1$; thus we see directly that $a\leq d-1$. Finally,  $H^1(Y; \scrO_Y(D+tF)) \neq 0$ $\iff$ $t\leq a-2$ and the linear system $|D+tF|$ has a base locus $\iff$ $t\leq a-1$.

We would now like to construct examples of genus one fibrations realizing all of the allowed numerical possibilities.

\begin{proposition}\label{order2} Let  $k$ be uncountable. Let $\sigma_0$ be the negative section of $\mathbb{F}_a$ and let $P$ be the class of a fiber.  Suppose that $a>0$ and  $\varepsilon \geq 1$  or that $a=0$ and $\varepsilon \geq 2$. Then, if $B$ a generic element of $|4\sigma_0 + 2(2a +\varepsilon)P|$, the corresponding double cover $Y$ of  $\mathbb{F}_a$ has no section. 
\end{proposition}

\begin{proof} Let $N = 2a+\varepsilon$. First note that, for all $N\geq 2a$, the linear system $|4\sigma_0 + 2NP|$ is base point free on $\mathbb{F}_a$. We consider degenerations $\mathcal{X}$ of $\mathbb{F}_a$ to $\mathbb{F}_a\cup \mathbb{F}_0$, glued along a fiber $P_0$, obtained by taking a trivial family $\mathbb{F}_a\times \mathbb{A}^1$  and blowing up a fiber of $\mathbb{F}_a\times \{0\}$ in the threefold $\mathbb{F}_a\times \mathbb{A}^1$. For the fiber $\mathbb{F}_a\cup \mathbb{F}_0$ over $0$, the  line bundle $\mathcal{L} = \pi_1^*\scrO_{\mathbb{F}_a}(4\sigma_0 + 2NP) \otimes \scrO_{\mathbb{F}_a\times \mathbb{A}^1}(-2\mathbb{F}_0)$ restricts on the copy of $\mathbb{F}_a$ to $\scrO_{\mathbb{F}_a}(4\sigma_0 + 2(N-1)P)$ and on the $\mathbb{F}_0$ component to $\scrO_{\mathbb{F}_0}(4\sigma_0 + 2P)$. Clearly $\mathcal{L}$ is divisible by $2$ in $\Pic \mathcal{X}$. If $\mathcal{B}$ is a smooth section of $\mathcal{L}$, the double cover $\mathcal{Y}$ of $\mathcal{X}$ branched along $\mathcal{B}$ fibers over $\mathbb{A}^1$. The  general fiber is a genus one fibration $Y\to \Pee^1$ with $\chi(\scrO_Y) = N-a=a + \varepsilon$ and the fiber over $0$ is of the form $Y_0 \cup R$, where $Y_0 \to \Pee^1$ is a genus one fibration with $\chi(\scrO_{Y_0}) = N-a-1=a + \varepsilon-1$, $R$ is a rational elliptic surface, and $Y_0$ and $R$ meet transversally along a smooth fiber $F_0$. If either $a\geq 1$ and $\varepsilon \geq 1$ or $a=0$ and $\varepsilon \geq 2$, $\chi(\scrO_{Y_0}) \geq 1$ and hence $\Pic Y_0$ is discrete. If there is a divisor of fiber degree one on the general fiber (or equivalently a section), then, possibly after a base change, there is a Cartier divisor on $Y_0 \cup R$ which restricts to a section of both $Y_0$ and $R$. If we have to make a base change of order $k$ around $0\in U$, then the threefold $\mathcal{Y}$  acquires a curve of $A_k$ singularities and is resolved by a smooth threefold $\widetilde{\mathcal{Y}}$ with central fiber $Y_0\cup E_1\cup \cdots \cup E_k \cup R$, where the $E_i$ are elliptic ruled, $E_i\cap E_{i+1}$ is a section in $E_i$ and in $E_{i+1}$, $E_1\cap Y =F_0$ and $E_k\cap R=F_0$. The argument in the general case is then essentially reduced to the case $\widetilde{\mathcal{Y}}=\mathcal{Y}$ and we shall just write down this case. There would thus be a divisor $\sigma_0$ on $Y_0$ and $\sigma_R$ on $R$, with $\sigma_0\cdot F = \sigma_R\cdot F=1$, and, for the common fiber $F_0$ of $Y_0$ and $R$, $\scrO_{Y_0}(\sigma_0)| F_0 = \scrO_R(\sigma_R)| F_0$. We shall show that, for generic choices, this is impossible.

For a fixed genus one fibration $Y_0$, the image of $\Pic Y_0$ in $\Pic F_0$ is a finitely generated abelian group $\Gamma$. On the other hand, fixing the elliptic curve $F_0$, it is easy to check the following: for every choice $p_1, \dots, p_9$ of points $p_i\in F_0$, there exists a rational elliptic surface $R$ and an inclusion of $F_0$ as a fiber of $R$, such that the image of $\Pic R$ in $\Pic F_0$ is generated by the line bundles $\scrO_{F_0}(p_i)$, together with the image of the line bundle $h$, which is a cube root of $\scrO_{F_0}(\sum_ip_i)$. (Embed $F_0$ in $\Pee^2$ by a linear system $h$ such that $3h \equiv \sum_ip_i$ and let $R$ be the surface which is $\Pee^2$ blown up at the images of the $p_i$.) It is then a straightforward exercise to see that, as long as $k$ is uncountable, for a fixed finitely generated subgroup $\Gamma$ of $\Pic F_0$ and for generic choices of the $p_1, \dots, p_9$, the intersection of the subgroup generated by the  $\scrO_{F_0}(p_i)$ and $h$ with $\Gamma$ is trivial. In particular the equality $\scrO_{Y_0}(\sigma_0)| F_0 = \scrO_R(\sigma_R)| F_0$ is impossible.

We must still show that we can find a sufficiently general rational elliptic surface as one component of the special fiber of a $\mathcal{Y}\to \mathbb{A}^1$ as constructed above. View $\mathcal{Y}$ as a double cover of $\mathcal{X}$ as above. Fix a smooth divisor $B_0$ in the linear system $|4\sigma_0 + 2(N-1)P|$ on $\mathbb{F}_a$ which meets a given fiber $f_0$ transversally in $4$ points; this is possible for all $N\geq 2a+1$. Given a smooth divisor $B_1$ in $|4\sigma_0 + 2P|$ on $\mathbb{F}_0$ such that $B_1 \cap P_0 = B_0\cap P_0$, the curve $B_0\cup B_1$ corresponds to a section $s$ of $\mathcal{L}|\mathcal{X}_0$. A straightforward argument shows that $H^1(\mathcal{X}_0; \mathcal{L}|\mathcal{X}_0) =0$, and hence that there is a Zariski open neighborhood $U$ of $0\in \mathbb{A}^1$ such that the section $s$ of $\mathcal{L}|\mathcal{X}_0$ extends to a section of $\mathcal{L}$ over $U$ (which we may assume to be smooth). The corresponding double cover is the degeneration we seek, provided that we can choose the section $B_1$ so that the rational elliptic surface $R$ is sufficiently generic.

To see that we can find a generic section $B_1$, it is sufficient to show that, if $R$ is a generic rational elliptic surface containing a fiber isomorphic to $F_0$, then $R$ is a double cover of $\mathbb{F}_0$, necessarily branched along a smooth curve in $|4\sigma_0 + 2P|$ (possibly after switching the order of the factors of $\mathbb{F}_0\cong \Pee^1\times \Pee^1$). It is enough to consider the case where $R$ is any rational elliptic surface with all fibers irreducible, or equivalently such that there does not exist a smooth rational curve on $R$ with self-intersection $-2$. In this case, we claim that there is a section $\tau$ of $R$ such that $\tau\cdot \sigma_R =1$. It is then easy to check that the linear system $|\sigma_R+\tau +F|$ has dimension $3$ and degree $4$, and defines a morphism $\varphi \colon R \to \Pee^3$ which is of degree $2$ onto its image, which is a smooth quadric in $\Pee^3$. Thus $\varphi$ realizes $R$ as a double cover of $\mathbb{F}_0$. To find $\tau$, use the fact that (as all fibers are irreducible) sections of $R$ correspond to elements of $\{\sigma_R, F\}^\perp \cong -E_8\subseteq \Pic R$, and that the condition that a section $\tau$ corresponding to $\alpha \in -E_8$ satisfy $\sigma_R \cdot \tau = 1$ is $\alpha^2=-4$. Since there do  exist vectors in $-E_8$ of square $-4$, we can  find $R$ as desired. This completes the proof.
\end{proof}

The following corollary shows that Corollary~\ref{P1ineqs} and Corollary~\ref{cor1.6} are sharp in the case of order $2$:

\begin{corollary} Let $d$ be an integer, $d\geq 2$. 
\begin{enumerate}
\item[\rm(i)] For all $a$, $0\leq a \leq d-1$, there exists a genus one fibration $Y\to \Pee^1$ of order $2$ which is a double cover  of $\mathbb{F}_a$.
\item[\rm(ii)] There exist a genus one fibration $Y\to \Pee^1$ of order $2$ and a divisor $D$ on $Y$ such that  $D\cdot F =2$, $h^1(D)\neq 0$ and $D^2 = 2d-10$.
\item[\rm(iii)] There exist a genus one fibration $Y\to \Pee^1$ of order $2$ and a divisor $D$ on $Y$ such that $D\cdot F =2$, $|D|$ has a base point and  $D^2 = 2d-6$.
\end{enumerate}
\end{corollary}
\begin{proof} Given $a$, $0\leq a \leq d-1$, take $Y$ to be a generic double cover of $\mathbb{F}_a$ branched along a section of $4\sigma_0 +2(2a+\varepsilon)P$, where $\varepsilon = d-a$. Note that, if $a=0$, then $\varepsilon \geq 2$. The induced morphism $f\colon Y \to \Pee^1$ realizes $Y$ as a genus one fibration (clearly without multiple fibers) with a divisor class $D = g^*\sigma_0$ such that $D\cdot F =2$. By Proposition~\ref{order2}, $Y$ has order $2$, and the discussion before the proof of Proposition~\ref{order2} shows that $\chi(\scrO_Y) =d$. This proves (i). Similarly, with $Y$ as in (i) and with $D = g^*\sigma_0$, consider  $D + (a-2)F= g^*(\sigma_0 + (a-2)P)$. Then  $h^1(D+ (a-2)F) \neq 0$ and $D^2 = -2a + 4 (a-2) = 2a-8 = 2d-10$. This proves (ii), and (iii) follows similarly by considering $D + (a-1)F$.
\end{proof}

\bigskip
\noindent \textbf{The case $n=3$:} Let $Y$ be a genus one fibration of order $3$. Again, we shall just look at the case $Y =\bar{Y}$, i.e.\ $D$ is $f$-ample. We can normalize $D$ so that the rank two vector bundle $f_*\scrO_Y(D)$ is equal to $\scrO_{\Pee^1}\oplus  \scrO_{\Pee^1}(-a)\oplus  \scrO_{\Pee^1}(-b)$, where $0\leq a \leq b$. The bounds of Corollary~\ref{P1ineqs} give $2b-a \leq d$ and $a+b\leq 2d-2$. 

Let $\Pee = \Pee(\scrO_{\Pee^1}\oplus \scrO_{\Pee^1}(a) \oplus \scrO_{\Pee^1}(b))$ with $\phi\colon \Pee\to \Pee^1$ the projection. There is a  surface $\Sigma_1\subseteq \Pee$ with $\phi|\Sigma_1$ a $\Pee^1$-bundle whose fibers are linearly embedded in the fibers of $\phi$ and such that $\phi_*\scrO_{\Pee}(\Sigma_1) = \scrO_{\Pee^1}\oplus \scrO_{\Pee^1}(-a) \oplus \scrO_{\Pee^1}(-b)$, and it is unique if $a>0$. By our assumptions, $g\colon Y \to \Pee$ embeds $Y$ as a smooth surface in the linear system $|3\Sigma_1 + NP|$, and we may choose $D$ so that $D = g^*\Sigma_1$ as divisor classes.

\begin{lemma} In the above notation, $d = N-(a+b)$. Hence $N\geq 3b$, and, if $a=b$, $N\geq 3b+1$.
\end{lemma}
\begin{proof} First, we have the following formula for the relative dualizing sheaf: 
$$\omega_{\Pee/\Pee^1} = \scrO_{\Pee}(-3\Sigma_1) \otimes \phi^*\det(\scrO_{\Pee^1}\oplus \scrO_{\Pee^1}(-a) \oplus \scrO_{\Pee^1}(-b)) = \scrO_{\Pee}(-3\Sigma_1-(a+b)P).$$
Thus $K_{\Pee} = \scrO_{\Pee}(-3\Sigma_1-(a+b+2)P)$. By adjunction $K_Y = \scrO_Y((N-(a+b+2))F$ and thus $d= N-(a+b)$. 

To see the last statement, by Corollary~\ref{P1ineqs} $N= a+b+d \geq a+b+(2b-a) = 3b$. If $a=b$, then $a+b =2b \leq 2d-2$, so that $d\geq b+1$ and 
$N= a+b+d \geq 3b+1$.
\end{proof}

\begin{lemma}\label{lemma4.8} If $D_t = (\Sigma_1 + tP)|Y= \Sigma_1|Y + tF$, then:
\begin{enumerate}
\item[\rm(i)] $D_t^2 = -2a-2b + d + 6t$.
\item[\rm(ii)] $H^1(Y; \scrO_Y(D_t)) \neq 0$ $\iff$  $t\leq b-2$. Hence $H^1(Y; \scrO_Y(D)) \neq0$ $\iff$ $D^2 \leq  d-2a+4b-12$.
\item[\rm(iii)] If $N\geq 3b+1$ and $a<b$, the linear system $|D_t|$ has a base locus for all $t\leq b-1$.
\end{enumerate}
\end{lemma}
\begin{proof} We begin with the following claim:

\begin{claim} With notation as above, $\Sigma_1 \cong  \mathbb{F}_{b-a}$ and, if $\sigma_0$ denotes the negative section of $\Sigma_1$, then 
$\scrO_{\Sigma_1}(\Sigma_1) = \scrO_{\Sigma_1}(\sigma_0 - af)$.
\end{claim}

\noindent\textit{Proof of the claim.}
Applying $\phi_*$ to  the exact sequence
$$0\to \scrO_{\Pee} \to \scrO_{\Pee}(\Sigma_1) \to \scrO_{\Sigma_1}(\Sigma_1)\to 0,$$
we get
$$0\to \scrO_{\Pee^1} \to \scrO_{\Pee^1}\oplus \scrO_{\Pee^1}(-a) \oplus \scrO_{\Pee^1}(-b) \to \phi_*\scrO_{\Sigma_1}(\Sigma_1)\to 0,$$
so that $\phi_*\scrO_{\Sigma_1}(\Sigma_1) = \scrO_{\Pee^1}(-a) \oplus \scrO_{\Pee^1}(-b)$ and $\Sigma_1 = \mathbb{F}_{b-a}$. If $\sigma_0$ is the negative section of $\Sigma_1$, then $\phi_*\scrO_{\Sigma_1}(\sigma_0) = \scrO_{\Pee^1}\oplus \scrO_{\Pee^1}(a-b)$. Since $\scrO_{\Sigma_1}(\Sigma_1)$ has degree one on each fiber of $\Sigma_1$, 
$$\scrO_{\Sigma_1}(\Sigma_1) = \scrO_{\Sigma_1}(\sigma_0 - af).\qed$$

Returning to the proof of Lemma~\ref{lemma4.8}, to see (i) it is enough to consider $D=D_0$. In this case $D^2 = \Sigma_1\cdot \Sigma_1 \cdot (3\Sigma_1+ NP) = (\sigma_0 -af)\cdot (3\sigma_0+(N -3a)f)$, where this last intersection is computed on $\Sigma_1$. Thus
$$D^2 = 3(a-b) -3a + N-3a = -3a -3b+N = -2a-2b+d,$$
establishing (i). Part (ii) then follows from Lemma~\ref{H1nonzero}.

Finally, to see (iii), note that, for $t< b$, the negative section $\sigma_0$ of $\Sigma_1$ is contained in the base locus of $|\Sigma_1 + tP|$ since $(\Sigma_1 + tP)|\Sigma_1= \sigma_0 +(t-a)f$. Now
$$Y\cdot \sigma_0 = (3\sigma_0+(N -3a)f)\cdot \sigma_0 = 3a-3b +N -3a = N-3b.$$ 
Thus, since $N\geq 3b+1$, $Y\cap\sigma_0 \neq \emptyset$ and hence the base locus of $|D+tF|$ contains the nonempty set $Y\cap\sigma_0$, by Lemma~\ref{4.1}. This completes the proof of (iii).
\end{proof}

\begin{remark} In the situation of (iii) above, if $N =3b$, then $Y\cdot \sigma_0 =0$. Thus either $Y\cap\sigma_0 =\emptyset$ or $\sigma_0 \subseteq Y$. If $Y\cap\sigma_0 =\emptyset$, as long as $t>a$, it is easy to check that $|D+tF|$ is base point free. If $\sigma_0 \subseteq Y$, then $\sigma_0$ is a section of $Y$ and hence $Y$ does not have order $2$.
\end{remark}

As in the case $n=2$, we now  construct examples of $Y$ of order $3$: 

\begin{proposition}\label{n=3}  
\begin{enumerate}
\item[\rm(i)] For $N\geq 2b-1$, the restriction homomorphism $$H^0(\Pee; \scrO_{\Pee}(3\Sigma_1+NP)) \to H^0(\Sigma_1; \scrO_{\Sigma_1}(3\Sigma_1+NP))$$ is surjective.
\item[\rm(ii)] Suppose that $b>a$. The base locus of $|3\Sigma_1 + (3b-1)P|$ is $\sigma_0$. For $N=3b-1$, there exist smooth surfaces $Y$ in the linear system $|3\Sigma_1 + (3b-1)P|$. Every such $Y$ contains $\sigma_0$, which is a section of the genus one fibration.
\item[\rm(iii)] For $N\geq 3b$, the linear system $|3\Sigma_1 + NP|$ is base point free and hence contains smooth surfaces $Y$. If moreover $b>a$ and  $k$ is uncountable, the generic surface $Y$ in $|3\Sigma_1 + NP|$ satisfies: every line bundle $L$ on $Y$ has $\deg_FL \equiv 0 \mod 3$, where $F$ is the divisor class $P|Y$.
\item[\rm(iv)] Suppose $a=b$ and  $k$ is uncountable. Further suppose either that $b>0$ and  $N \geq 3b+1$ or that $b=0$ and $N\geq 2$. Then  the generic surface $Y$ in $|3\Sigma_1 + NP|$ satisfies: every line bundle $L$ on $Y$ has $\deg_FL \equiv 0 \mod 3$.
\end{enumerate}
\end{proposition}
\begin{proof} To prove (i), it suffices to show that $H^1(\Pee; \scrO_{\Pee}(2\Sigma_1+NP)) =0$. Now $R^1\phi_*\scrO_{\Pee}(2\Sigma_1+NP) =0$ and $$R^0\phi_*\scrO_{\Pee}(2\Sigma_1+NP) =\operatorname{Sym}^2(\scrO_{\Pee^1}\oplus \scrO_{\Pee^1}(-a) \oplus\scrO_{\Pee^1}(-b))\otimes \scrO_{\Pee^1}(N).$$
The most negative term in the direct sum is $\scrO_{\Pee^1}(-2b+N)$. Thus $$H^1(\Pee^1; R^0\phi_*\scrO_{\Pee}(2\Sigma_1+NP)) = 0$$ as long as $N\geq 2b-1$, and (i) follows from  the Leray spectral sequence.

Clearly  $3b-1 \geq 2b-1$. Note that $|3\Sigma_1 + NP|_{\Sigma_1} = 3\sigma_0 -3af+Nf$, which is base point free $\iff$ $N\geq 3b$. Thus, by (i), if $N\geq 3b$, then $|3\Sigma_1 + NP|$ is base point free and hence the generic $Y\in |3\Sigma_1 + NP|$ is smooth. This proves the first sentence in (iii). For (ii), $3\sigma_0 -3af+(3b-1)f= \sigma_0 + (2\sigma_0 + (3b-3a-1)f)$. This is of the form $\sigma_0 + $ base point free linear system, since $3b-3a-1 \geq 2b-2a$ (because $b-a\geq 1$).  To check that the generic $Y$ is smooth one can make a local computation along the base locus; details omitted.

To see the second statement of (iii), we must show that, for $N\geq 3b$ and $b>a$, the generic $Y$ has no divisor whose  fiber degree is positive and less than $3$, or equivalently the genus one fibration on $Y$ has no section.  The proof will be via degenerations, beginning with the following construction: Let $Y_0\in |3\Sigma_1 + (N-1)P|$ be a smooth surface, guaranteed to exist by (ii) and the first part of (iii) for all $N\geq 3b$. Note that 
$$\chi(\scrO_{Y_0}) = N-1-a-b\geq 3b -1-a-b  = b+ (b-a) -1 \geq b >0.$$
Thus $\Pic Y_0$ is discrete.
Choose a smooth fiber $F$ on $Y$ and let $\Pee^2$ be the fiber of $\Pee$ corresponding to $F$. There is the normal crossings surface $Y_0\cup \Pee^2$, whose double curve is $F$, a cubic in $\Pee^2$. Let $Y$ be a smooth surface in $|3\Sigma_1 + NP|$. Then we can make a threefold $\bar{Z}$ corresponding to the total space of the pencil spanned by $Y$ and $Y_0\cup \Pee^2$. The threefold $\bar{Z}$ is singular exactly at the intersection $Y\cap F$. If this intersection is transverse, then the singularities of $\bar{Z}$ will be nine threefold ordinary double points, since $Y\in |3\Sigma_1 + NP|$ meets the $\Pee^2$ in a  plane cubic. As is well known, there is a small projective resolution $Z$ of $\bar{Z}$ whose central fiber is $Y_0\cup R$, where $R$ is the rational elliptic surface obtained by blowing up $\Pee^2$ along the nine points $Y\cap F$. We next show that, for $N\geq 3b$, the nine points on $F$ are arbitrary subject to the condition that they are cut out by a plane cubic:

\begin{claim} If $N\geq 3b$, then the restriction map $H^0(\Pee; \scrO_{\Pee}(3\Sigma_1 + NP)) \to H^0(\Pee^2; \scrO_{\Pee^2}(3))$ is surjective.
\end{claim}
\begin{proof} We show that $H^1(\Pee; \scrO_{\Pee}(3\Sigma_1 + (N-1)F))=0$. As usual, apply Leray. Clearly $R^1\phi_*\scrO_{\Pee}(3\Sigma_1 + (N-1)F) = 0$. Moreover, 
$$R^0\phi_*\scrO_{\Pee}(3\Sigma_1 + (N-1)F) = \operatorname{Sym}^3(\scrO_{\Pee^1}\oplus \scrO_{\Pee^1}(-a) \oplus\scrO_{\Pee^1}(-b))\otimes \scrO_{\Pee^1}(N-1).$$
The most negative term that appears is $\scrO_{\Pee^1}(-3b+N-1)$. To guarantee that $H^1(R^0\phi_*\scrO_{\Pee}(3\Sigma_1 + (N-1)F))=0$, it suffices to take $-3b+N-1 \geq -1$, i.e.\ $N\geq 3b$.
\end{proof}

Returning to the proof of the lemma, we may assume that $R$ is the blowup of $\Pee^2$ at nine general points $p_1, \dots, p_9\in F$, subject to the condition that $p_1+\dots + p_9\in |\scrO_{\Pee^2}(3)|_F|$. Equivalently, we can choose the eight points $p_1, \dots, p_8$ arbitrarily and then the ninth is determined. 

Replacing the base $\Pee^1$ of the pencil by a Zariski open subset $\Delta$, we may assume that $f\colon Z \to \Delta$ is projective, with all fibers smooth except the fiber over $0$, which is $Y_0\cup R$. If for all $t$ the fiber $f^{-1}(t)$ has a section, then perhaps after a base change there is a line bundle $\mathcal{L}$ over $Z$ such that $\mathcal{L}|f^{-1}(t)$ has fiber degree one.  The argument in  case we need to make a base change is then essentially reduced to the case $\widetilde{Z}=Z$ (as in our discussion of the case $n=2$) and we shall just write down this case.

 Consider $\mathcal{L}|(Y_0\cup R)$. This corresponds to line bundles $L_0$ on $Y_0$ and $L_1$ on $R$ whose restrictions to $F$ agree. Let $G$ be the finitely generated subgroup of $\Pic F$ which is the image of the restriction map $\Pic Y_0\to \Pic F$. Then there is an equation in $\operatorname{Div} F$ of the form
$$q \equiv ah + \sum _{i=1}^8n_ip_i.$$
Here $q\in F$ corresponds to some element in $G$ of degree one, $h \in |\scrO_{\Pee^2}(3)|_F| $, and the $n_i\in \Zee$.  Since $h$ has degree $3$ and $q$ has degree one, the $n_i$ are not all $0$. Equivalently, $\sum _{i=1}^8n_ip_i$ is a nonzero element lying in some fixed finitely generated subgroup of $F$. But since $F$ is uncountable, it is clearly possible to choose the $p_i$ so that this does not happen. This is a contradiction. Hence the generic $Y$ has no section.

The proof of (iv) is similar.
\end{proof}

\begin{remark} In the situation of (iv), with $a=b$ and $N=3b$, $Y \in |3\Sigma_1 + 3bF|$, $\Sigma_1\cong \mathbb{F}_0$, and $Y\cdot \mathbb{F}_0=3\sigma_0$. Since every element in $|3\sigma_0|$ is a sum of three elements in $|\sigma_0|$, $Y$ contains a section.
\end{remark}

We now show that Corollary~\ref{P1ineqs} and Corollary~\ref{2cor4} are sharp in the case of order $3$.

\begin{corollary} Let  $d$ be an integer, $d\geq 2$.  
\begin{enumerate}
\item[\rm(i)] Let $a$ and $b$ be two integers, $0\leq a \leq b$, such that $2b-a\leq d$ and $a+b\leq 2d-2$. There exist a genus one fibration $Y\to \Pee^1$ with $\chi(\scrO_Y) =d$ of order $3$ which is embedded in $\Pee= \Pee(\scrO_{\Pee^1}\oplus \scrO_{\Pee^1}(a) \oplus \scrO_{\Pee^1}(b))$ as a divisor of relative degree $3$.
\item[\rm(ii)]  There exist a genus one fibration $Y\to \Pee^1$ with $\chi(\scrO_Y) =d$ of order $3$ and a divisor $D$ on $Y$ with $D\cdot F =3$, $h^1(D)\neq 0$ and $D^2 = 3d-12$.
\item[\rm(iii)] Suppose that $d\geq 3$. There exist a genus one fibration $Y\to \Pee^1$ with $\chi(\scrO_Y) =d$ of order $3$ and a divisor $D$ on $Y$ with $D\cdot F =3$,  such that $|D|$ has a base point, and $D^2 = 3d-8$.
\end{enumerate}
\end{corollary}
\begin{proof} To see (i), first suppose that $a>b$. In this case, since $2b-a = b+ (b-a) \leq d$, $b\leq d-1$, and hence the inequality $a+b\leq 2d-2$ is a consequence of the inequality $2b-a\leq d$. Set $N = a+b+d \geq (a+b) + (2b-a) = 3b$. By (iii) of Proposition~\ref{n=3}, there exists a smooth $Y \in |3\Sigma_1 + NP|$ with no section and with $\chi(\scrO_Y) =d$. The projection $\Pee \to \Pee^1$ realizes $Y$ as a genus one fibration (with no multiple fibers), and $D = \Sigma_1|Y$ satisfies $D\cdot F=3$. Thus $Y$ has order $3$.

In case $a=b$, the inequalities $2b-a\leq d$ and $a+b\leq 2d-2$ reduce to $b\leq d-1$. In this case set $N = a+b+d \geq 2b + b+1 = 3b+1$, and $N = d\geq 2$ if $a=b=0$. The argument then concludes as in the case $a<b$ but using (iv) of Proposition~\ref{n=3}. 

To see (ii), choose $a< b$ such that $d= 2b-a$, which is possible as long as $d\geq 2$,  Let $Y$ be as in (i), and consider the divisor $D=\Sigma_1|Y + (b-2)F = D_{b-2}$, in the notation of (ii) of Lemma~\ref{lemma4.8}, then $h^1(D) \neq 0$ and $D^2 = d + 4b-2a -12 = 3d -12$. Finally, to see (iii), choose $a< b$ such that $d= 2b-a+1$ and let $Y$ be as in (i). Then $N = a+b+d =  3b+1$. By (iii) of Lemma~\ref{lemma4.8}, if $D =\Sigma_1|Y + (b-1)F = D_{b-1}$, the linear system $|D|$ has a base point, and $D^2 = d + 4b-2a -6= 3d-8$.
\end{proof}

\begin{remark} In case $d=2$ in (iii) above, $Y$ is a $K3$ surface and one can construct examples of divisors $D$ with $D$ effective, $D\cdot F =3$ and $D^2 =-2$ directly, so that $|D|$ has a fixed component. It suffices to take $a=b=1$,  $N=4$, and $D = \Sigma _1\cdot Y$.
\end{remark}

\section{Irreducible components of the moduli space}

Let $k$ be algebraically closed and let $n$ be a positive integer prime to the characteristic of $k$.  Later, in the discussion of Pontrjagin squares, we will further have to assume that $\operatorname{char} k \neq 2$, and eventually that $k=\Cee$. Our goal in this section is to describe the set of components of the moduli space of genus one fibrations over $\Pee^1$. We shall primarily be concerned with the case where all fibers of $\pi$ are irreducible, because of the following:

\begin{theorem}\label{deform} Suppose that $k=\Cee$. Let $f\colon Y \to \Pee^1$ be a genus one fibration and let $D$ be a divisor on $Y$ of fiber degree $n$. Then there exists a complex manifold $\mathcal{Y}$, a flat proper morphism $\Phi\colon \mathcal{Y}\to \Pee^1\times \Delta$, where $\Delta$ is the unit disk in $\Cee$, and a divisor $\mathcal{D}$ on $\mathcal{Y}$ such that, if $\Psi\colon \mathcal{Y}\to \Delta$ is the composition, then $\Psi^{-1}(0) = Y_0 = Y$, $\Psi^{-1}(t) = Y_t$ is a smooth genus one fibration over $\Pee^1$ with all fibers irreducible for all $t\neq 0$, and $\mathcal{D}$ induces a divisor on $Y_t$ for all $t$, such that $\mathcal{D}|Y_0$ has the same restriction to the generic fiber as $D$. \qed 
\end{theorem}
\begin{proof} (Sketch.) 
The proof of this theorem uses standard results going back to Kodaira (see also Kas \cite{Kas} or \cite[Section 1.5]{FriedMorg} for more details). Let $\mathcal{M}_d$ be the coarse moduli space of elliptic surfaces $\pi \colon X \to \Pee^1$ with $\chi(X;\scrO_X) = d$. (One has to be a little  careful in case $d=1$.) Then $\mathcal{M}_d$ is irreducible, since via the existence of Weierstrass models there is a surjection $U \to \mathcal{M}_d$, where $U$ is a nonempty Zariski open subset in the product $|\scrO_{\Pee^1}(4d)|\times |\scrO_{\Pee^1}(6d)|$. In particular, beginning with an elliptic surface $\pi \colon X \to \Pee^1$, there exists a smooth curve, which we may take to be the unit disk $\Delta$, and an elliptic fibration with a section $\mathcal{X} \to \Delta\times \Pee^1$ such that, if $\Pi\colon \mathcal{X}\to \Delta$ is the induced map, then $\Pi^{-1}(0)\cong X$ and $\Pi^{-1}(s)$ is an elliptic surface all of whose fibers are irreducible for $s\neq 0$. Here, one needs the existence of simultaneous resolution of surface rational double points since the Weierstrass model of an elliptic surface with reducible fibers will be singular.

If $\pi\colon X \to B$ is an elliptic surface, there is an analytic Tate-Shafarevich group $H_{\text{an}}^1(B; \mathcal{E})$ and a surjective homomorphism $H^2(X; \scrO_X) \to H_{\text{an}}^1(B; \mathcal{E})$ \cite[Section 1.5, Lemma 5.11]{FriedMorg}. Let $\Pi\colon \mathcal{X}\to S$ be an arbitrary family of elliptic surfaces over a  complex manifold $S$. Then a point of the total space $\mathbb{V} = \mathbb{V}(R^2\Pi_*\scrO_{\mathcal{X}})$ corresponds to a complex elliptic surface $Y_s$ whose Jacobian surface is isomorphic to $X_s =\Pi^{-1}(s)$. Moreover, there is a tautological family of genus one fibrations $\Psi\colon \mathcal{Y} \to \mathbb{V}$ \cite[Section 1.5, Proposition 5.31]{FriedMorg}. Specializing to the case $B=\Pee^1$, let $f\colon Y \to \Pee^1$ and  $D$ be as in the statement of the theorem,   let $X$ be the Jacobian surface of $Y$, and let $\mathcal{X} \to \Delta$ be a one-parameter deformation of $X$ as in the preceding paragraph.  Thus there is a point $v_0\in \mathbb{V} = \mathbb{V}(R^2\Pi_*\scrO_{\mathcal{X}})$ corresponding to $Y$, and a tautological family $\Psi\colon \mathcal{Y} \to \mathbb{V}$. Let $\xi \in H^2(Y;\Zee)$ be the class $[D]$. Then, as $\mathbb{V}$ is simply connected, $\xi$ extends to a section $\sigma$ of $R^2\Psi_*\Zee$ over $\mathbb{V}$, and the locus $\Xi$ where $\xi$ is of type $(1,1)$ corresponds to the set of points  $v\in \mathbb{V}$ where the image of $\sigma$ in the fiber of $R^2\Psi_*\scrO_{\mathcal{Y}}$ over $v$ is $0$. Hence $\Xi$ is an analytic subset of $\mathbb{V}$ and every component of $\Xi$ has codimension at most $h^{2,0}(Y) = h^{2,0}(X) = d-1$ and hence dimension at least one.  On the other hand, it is easy to see that the set of algebraic genus one fibrations parametrized by $H^2(X; \scrO_X)$ is countable. Thus, $\Xi \cap \mathbb{V}(0)$ is discrete and the induced morphism $\Xi \to \Delta$ is nonconstant on every component. Choosing the normalization of a component of $\Xi$ passing through $v_0$, which we can assume is isomorphic to a disk, and replacing $\Psi$ by its restriction to this component, then gives a deformation of $Y$ as desired. 
\end{proof}

For arbitrary fields $k$, we can only show that $Y$ can be deformed to a genus one fibration $Y_t$ over $\Pee^1$ with all fibers irreducible and such that there exists a divisor on $Y_t$ whose fiber degree is divisible by $n$. However, a stronger result is most likely true in positive characteristic as well.
 \textbf{In what follows,  we assume that all fibers of $\pi$ are irreducible unless explicitly stated.}

More generally let $\pi \colon X \to B$ be an elliptic surface (with a fixed section). The  set of triples $(Y,f, \varphi)$, where $f\colon Y \to B$ is a genus one fibration and $\varphi\colon J(Y)\to X$ is an isomorphism from the Jacobian surface of $Y$ to $X$ over $B$ sending the zero section of $J(Y)$ to the fixed section of $X$, can be identified with the group $H^1(B; \mathcal{E})$, where $\mathcal{E}$ is the sheaf of sections of $\pi\colon X \to B$ (and cohomology means \'etale cohomology). We denote the class of the triple $(Y,f, \varphi)$ in $H^1(B; \mathcal{E})$ by $[Y, \varphi]$. There is a split exact sequence
$$0 \to \mathcal{E} \to R^1\pi_*\mathbb{G}_m/\mathcal{T} \xrightarrow{\deg} \Zee\to 0,$$
where $\mathcal{T}$ is generated by the line bundles associated to components of fibers and the homomorphism $R^1\pi_*\mathbb{G}_m/\mathcal{T} \to \Zee$ is essentially given by taking degree along the fibers. In particular, if all fibers of $\pi$ are irreducible, the sheaf $\mathcal{E}$ may be identified with a subsheaf of  $R^1\pi_*\mathbb{G}_m$. As $B$ is a curve, the principal homogeneous space $Y$ is trivial if and only if its restriction to the generic point $\Spec K$ of $B$ is trivial (since a point over $\Spec K$ automatically extends to a section of $f$). From this, it is easy to see that the point $\xi\in H^1(B; \mathcal{E})$ corresponding to $(Y, f, \varphi)$ is an $n$-torsion point if and only if there exists a divisor $D$ on $Y$ with $D\cdot F = n$. Thus the order of $Y$ as defined in the introduction is the same as the order of the corresponding element of $H^1(B; \mathcal{E})$.

There is an analogous construction for $B$ an excellent Noetherian scheme and $\pi\colon \mathcal{X} \to B$  a proper flat morphism, not necessarily smooth,  all of whose geometric fibers are irreducible of arithmetic genus one, with a section $\sigma$ whose image is contained in the smooth locus of $\pi$ (cf.\ for example \cite{Katz-Mazur}, \cite{ASD},
\cite{dejong}). Explicitly, if  $\mathcal{E}$ is the smooth locus of $\pi$, then there is a canonical identification of group schemes $\mathcal{E}\cong \underline{\Pic}^0_{\mathcal{X}/B}$ which maps a section $\tau$ of $\mathcal{E}\to B$ to the line bundle $\scrO_{\mathcal{X}}(\tau-\sigma)$. A straightforward argument using the existence of Weierstrass equations or the compactification of the relative Picard scheme (\cite{AltmanKleiman} or \cite[Section 8.2]{BoschLutkebohmertRaynaud}) shows that the action of $\mathcal{E}$ on itself by multiplication extends to an action $\mathcal{E}\times_B\mathcal{X}\to \mathcal{X}$. Now suppose that $f\colon \mathcal{Y} \to B$ is a proper flat morphism,  all of whose geometric fibers are irreducible of arithmetic genus one, and that we are given an isomorphism of group schemes $\varphi\colon \underline{\Pic}^0_{\mathcal{Y}/B} \to \mathcal{E}$. Then, as shown in \cite{ASD}, if $\mathcal{Y}^0$ is the smooth locus of $f$, then $\mathcal{Y}^0$ is a principal homogeneous space for $\mathcal{E}$. Denote by $[\mathcal{Y}, \varphi]$ the corresponding element of $H^1(B; \mathcal{E})$. Given $\varphi$, we can construct a compactification $\mathcal{Y}'$ of $\mathcal{Y}^0$ by taking the associated space $\mathcal{Y}^0\times ^{\mathcal{E}}\mathcal{X}$ which is the quotient of $\mathcal{Y}^0\times _{B}\mathcal{X}$ by the antidiagonal action of $\mathcal{E}$. In the geometric case, where $B$ is a smooth curve over $k$, $\mathcal{Y}\cong \mathcal{Y}'$,  since both surfaces have unique relatively minimal models. In general, we shall always make the assumption that $\mathcal{Y}\cong \mathcal{Y}'$, i.e.\ that $\mathcal{Y}$ is determined by (and determines) the class $[\mathcal{Y}, \varphi]\in H^1(B; \mathcal{E})$, given the fixed compactification $\mathcal{X}$ of $\mathcal{E}$.

Suppose that $n$ is a positive integer invertible on $B$. The $n$-torsion points in $H^1(B; \mathcal{E})$ can be described as follows: begin with the exact sequence (of sheaves on $\mathcal{X}$)
$$ \{1\} \to \mmu_n \to \mathbb{G}_m\xrightarrow{\times n} \mathbb{G}_m \to \{1\}.$$
Since $\pi_*\mathbb{G}_m = \mathbb{G}_m$, the homomorphism $R^1\pi_*\mmu _n \to R^1\pi_*\mathbb{G}_m$ is injective. It follows that there is an injection $R^1\pi_*\mmu _n \to \mathcal{E}$ whose image is the kernel of multiplication by $n$, and hence a homomorphism $H^1(B; R^1\pi_*\mmu _n) \to H^1(B; \mathcal{E})$ whose image is in the subgroup of $n$-torsion in $H^1(B; \mathcal{E})$. This image is actually equal to the $n$-torsion because multiplication by $n$ on $\mathcal{E}$ is surjective (see below).

One can also proceed as follows: define $\mathcal{E}[n]$ via the exact sequence
$$0 \to \mathcal{E}[n] \to \mathcal{E}\xrightarrow{\times n}  \mathcal{E} \to 0,$$
where we have used the assumption that the fibers of $\pi$ are irreducible to conclude that multiplication by $n$ is surjective.
The above arguments show  that  $\mathcal{E}[n] \cong R^1\pi_*\mmu _n$.  Moreover there is an induced homomorphism $H^1(B; \mathcal{E}[n]) \to H^1(B; \mathcal{E})$ whose image is the subgroup of $n$-torsion points in $H^1(B; \mathcal{E})$. However the homomorphism $H^1(B; \mathcal{E}[n]) \to H^1(B; \mathcal{E})$  is not always injective. Artin and Swinnerton-Dyer have proved (\cite{ASD}, Proposition (1.7)):

\begin{lemma}\label{ASDlemma} The group $H^1(B; \mathcal{E}[n])$ classifies the set of isomorphism classes of quadruples $(\mathcal{Y}, f, \varphi, \tau)$, where $f\colon \mathcal{Y} \to B$ is a flat proper morphism, all of whose geometric fibers have arithmetic genus one, $\varphi\colon \underline{\Pic}^0_{\mathcal{Y}/B} \to \mathcal{E}$ is an isomorphism of group schemes, with $\mathcal{Y}\cong \mathcal{Y}^0\times ^{\mathcal{E}}\mathcal{X}$ as schemes over $B$, and $\tau$ is a section of $\underline{\Pic}_{\mathcal{Y}/B}$ over $B$ of degree $n$. Here an isomorphism $\Psi\colon (\mathcal{Y}, f, \varphi, \tau) \to (\mathcal{Y}', f', \varphi', \tau')$ consists of an isomorphism $\Psi \colon \mathcal{Y}^0 \to (\mathcal{Y}')^0$ such that $f'\circ \Psi = f$, $\varphi' \circ \Psi_* = \varphi$, where $\Psi_*$ is the naturally induced isomorphism from $\mathcal{E}$ to $\mathcal{E}'$, and $\Psi^*\tau'=\tau$. \qed
\end{lemma}

\begin{corollary} Suppose that $B$ is a smooth curve over $k$. Then the group $H^1(B; \mathcal{E}[n])$ classifies the set of isomorphism classes of quadruples $(Y, f, \varphi, D)$, where $f\colon Y \to B$ is a genus one fibration, $\varphi$ is an isomorphism of elliptic surfaces with a section from $J(Y)$ to $X$, and $D$ is a divisor on $Y$ such that $D\cdot F = n$, and an isomorphism $\Psi\colon (Y, f, \varphi, D) \to (Y', f', \varphi', D')$ consists of an isomorphism $\Psi \colon Y \to Y'$ such that $f'\circ \Psi = f$, $\varphi' \circ \Psi_* = \varphi$, where $\Psi_*$ is the naturally induced isomorphism from $J(Y)$ to $J(Y')$, and $\Psi^*D'$ has the same restriction to the generic fiber of $f$ as $D$. \qed
\end{corollary}

In the surface case, given a quadruple  $(Y, f, \varphi, D)$ as above, a concrete way to determine the class $[Y, \varphi, D]=\alpha \in H^1(B; \mathcal{E}[n]) = H^1(B; R^1\pi_*\mmu_n)$ given by  Lemma~\ref{ASDlemma} is via the following description of the filtration  on $H^2(X; \mmu_n)$ induced by the  Leray spectral sequence. Let $\rho \colon H^2(X; \mmu_n) \to H^0(B; R^2\pi_*\mmu_n)$ be the usual homomorphism. The pullback $H^2(B; \mmu_n) \to H^2(X; \mmu_n)$ has image equal to $\Zee/n\Zee\cdot [F]$, which is contained in $\Ker \rho$, and  
$$H^1(B; R^1\pi_*\mmu _n)\cong \Ker \rho/ (\Zee/n\Zee\cdot [F]).$$
 As all fibers of $\pi$ are irreducible, the   homomorphism 
 $$H^2(X; \mmu_n) \to H^0(B; R^2\pi_*\mmu_n)\cong \Zee/n\Zee$$
    sends $\alpha$ to $\alpha([F])$, and hence $H^1(B; R^1\pi_*\mmu _n)\cong [F]^\perp/(\Zee/n\Zee\cdot [F])$.  Similar statements hold when we replace $\pi\colon X \to B$ with $f\colon Y \to B$. If $D$ is a divisor on $Y$ with $D\cdot F = n$, then $D$ defines a class $[D]$ in $H^2(Y; \mmu_N)$ for any $N$. Taking $N=n$, we see that the image of $[D]$ in $H^0(B; R^2f_*\mmu_n)$ is $0$ and hence $D$ defines a class in $H^1(B; R^1f_*\mmu_n)$, also denoted $[D]$.

\begin{lemma}\label{identify}
For every $N$ prime to the characteristic and for all $i$, there is an isomorphism
$$ R^if_*\mmu_N  \cong R^i\pi_* \mmu_N$$
which respects the cup product.  
Via the identification $R^1f_*\mmu_n  \cong R^1\pi_* \mmu_n$, the element $-[D] \in H^1(B; R^1f_*\mmu_n)$ corresponds to the class $\alpha = [Y, \varphi, D] \in H^1(B; R^1\pi_*\mmu_n)$ given by Lemma~\ref{ASDlemma}. 
\end{lemma} 
\begin{proof}  The first statement follows from the fact that there is an \'etale open cover of $B$, say $\{U_i\}$, such that $\pi^{-1}(U_i) \cong f^{-1}(U_i)$ as schemes over $U_i$, and such that the transition functions are given by fiberwise translation and thus induce canonical isomorphisms in  cohomology. (This result does not need the assumption that all fibers of $\pi$ are irreducible.) 

To see the last statement, if $\{U_i\}$ is an \'etale open cover of $B$, we denote the pullback of $\sigma$ to $U_i$ by $\sigma_i$ and the pullback of $\sigma$ to $U_i\times _BU_j$ by $\sigma_{ij}$.  The  class $\alpha = [Y, \varphi, D]$ given in \cite[Proposition 1.7]{ASD} is defined as follows:  the group $H^1(B; R^1\pi_*\mmu_n) = H^1(B; \mathcal{E}[n])$ is the first hypercohomology group $\mathbb{H}^1(B;\mathcal{E}\xrightarrow{\times n}\mathcal{E})$, and as such corresponds in \v Cech cohomology to a pair $(\xi, \delta)$, where $\xi=\{\xi_{ij}\}$ is a $1$-cocycle  for $\mathcal{E}$, $\delta = \{\delta_i\}$ is a $0$-chain, and $n\xi_{ij} = \delta_i-\delta_j$. To pass from this description to an element of $H^1(B; \mathcal{E}[n])$, write $\delta_i = n\gamma_i$ for some section $\gamma_i$ of $X_{U_i}$, which is possible after refining the cover, and then $\xi_{ij} - \gamma_i+ \gamma_j$ is an $n$-torsion section of $\mathcal{E}$ over $U_i\times _BU_j$, corresponding to the $n$-torsion line bundle 
$$\scrO_{X_{U_i\times_BU_j}}(\xi_{ij} - \gamma_i+ \gamma_j -\sigma_{ij}),$$
which is then the corresponding representative for $(\xi,\delta)$.  (Here and in what follows all equalities of line bundles are modulo $\Pic (U_i)$ or $\Pic (U_i\times _BU_j)$ as appropriate.) Given $Y$ and $D$, the class $\alpha=(\xi, \delta)$ is defined as follows: We may assume that there exist sections $\tau_i$ on $Y_{U_i}$. The isomorphism $\varphi$ then uniquely defines an isomorphism of pairs $\varphi_i \colon (Y_{U_i}, \tau_i) \to (X_{U_i}, \sigma_i)$, and $\varphi_i\circ (\varphi_j)^{-1}$ is translation on $X_{U_i\times _BU_j}$ by a section $\xi_{ij}$. We can write $\scrO_{Y_{U_i}}(D) = \varphi_i^*\scrO_{X_{U_i}}(\delta_i+(n-1)\sigma_i)$ for a unique section $\delta_i$. Thus 
\begin{align*}(\varphi_i\circ(\varphi_j^{-1}))^*(\delta_i+(n-1)(\sigma_{ij})) &= \delta_i+ (n-1)(\sigma_{ij}) + n\xi_{ij} \\
&=\delta_j+(n-1)(\sigma_{ij}),
\end{align*}
and so   $\delta_i - \delta_j = n\xi_{ij}$ on $U_i\times _BU_j$. After passing to a refinement of the cover, we can write $\delta_i = n\gamma_i$ as sections of $\mathcal{E}(U_i)$; as divisors, $\delta_i =n\gamma_i -(n-1)\sigma_i$, and thus $\delta_i +(n-1)\sigma_i=n\gamma_i$. It follows that $\varphi_i^*\scrO_{X_{U_i}}(\gamma_i)$ is an $n^{\text{th}}$ root of $\scrO_{Y_{U_i}}(D)$.

On the other hand, the class $[D]\in H^2(Y; \mmu_n)$ is the image of the line bundle $\scrO_Y(D)\in H^1(Y; \mathbb{G}_m)$ under the Kummer exact sequence. There is also the exact sequence
$$0\to R^1f_*\mmu_n \to R^1f_*\mathbb{G}_m \to \mathcal{R} \to 0,$$
where $\mathcal{R}$ is the image in $R^1f_*\mathbb{G}_m$ of multiplication by $n$ from $R^1f_*\mathbb{G}_m$ to $R^1f_*\mathbb{G}_m$. Under the assumption that all fibers are irreducible, it is easy to see that $\scrO_Y(D)$ defines a class in $H^0(B;\mathcal{R})$, i.e.\ that there exists an \'etale open cover $\{U_i\}$ of $B$ such that $\scrO_Y(D)|Y_{U_i} = \mathcal{L}_i^{\otimes n}$.  So we can write $\scrO_Y(D)|Y_{U_i} = \mathcal{L}_i^{\otimes n}$, where $\mathcal{L}_i$ has degree one on every fiber. The line bundle $\mathcal{L}_i\otimes \mathcal{L}_j^{-1}$ on $Y_{U_i\times _BU_j}$ is then a line bundle which is $n$-torsion on every fiber and so corresponds to a section of $H^1(Y_{U_i\times _BU_j};\mmu_n)$. In this way we get a \v Cech cocycle which defines an element of $H^1(B; R^1f_*\mmu_n)$ corresponding to $[D]$. But as we have seen, we can take the $n^{\text{th}}$ root  $\mathcal{L}_i$ of $\scrO_Y(D)|Y_{U_i}$ to be the line bundle $(\varphi_i^{-1})^*\scrO_{X_{U_i}}(\gamma_i)$ in the notation of the preceding paragraph. Hence the line bundle $\mathcal{L}_i\otimes \mathcal{L}_j^{-1}$ on $Y_{U_i\times _BU_j}$ corresponds to the line bundle $$\scrO_{X_{U_i\times_BU_j}}(\gamma_i -(\gamma_j +\xi_{ij}-\sigma_{ij}).$$ Up to sign, this is the same as the class $[Y, \varphi, D]$.
\end{proof}

\begin{remark} One can show the following: suppose that $\pi \colon \mathcal{X} \to B$ is a proper flat morphism with all fibers of $\pi$ irreducible curves of arithmetic genus  one, where $B$ is an excellent scheme of finite dimension and $n$ is invertible in $\scrO_B$. Let $f\colon \mathcal{Y} \to B$ is a proper flat morphism,   all of whose geometric fibers are irreducible of arithmetic genus one, and that  $\varphi\colon \underline{\Pic}^0_{\mathcal{Y}/B} \to \mathcal{E}$ is an isomorphism of group schemes. Suppose that the class $[\mathcal{Y}, \varphi]\in H^1(B; \mathcal{E})$ is divisible by $n$ in the following sense: for every $m\geq 1$, there exists a $\xi \in H^1(B; \mathcal{E})$ such that $[\mathcal{Y}, \varphi] = n^m\xi$. Finally, let $N$ be an integer invertible in $\scrO_B$.

Then $\mathbb{R}f_*\mmu_N \cong \mathbb{R}\pi_*\mmu_N$ (respecting the cup product), which implies both that $R^if_*\mmu_N  \cong R^i\pi_* \mmu_N$, by taking cohomology sheaves, and that  $H^i(Y; \mmu_N) \cong H^i(X; \mmu_N)$, by taking hypercohomology. This is the analogue of the situation for $k=\Cee$, where in fact $X$ and $Y$ are diffeomorphic.

It is not clear if this result continues to hold if we drop the divisibility assumption on the class of $Y$ in $H^1(B; \mathcal{E})$, for example if $B$ is a curve over a finite field. 

Using the exact sequence $0\to \mathcal{E}[n] \to \mathcal{E} \to \mathcal{E} \to 0$ and the identification $\mathcal{E}[n] \cong R^1\pi_*\mmu_n$, we see that the hypothesis on the divisibility of $[\mathcal{Y}, \varphi]$ is satisfied if $H^2(B; R^1\pi_*\mmu_n) =0$. If for example $\pi\colon X \to B$ is a smooth elliptic surface with all fibers irreducible and at least one singular fiber, then one can check that $H^2(B; R^1\pi_*\mmu_n) =0$ (cf.\ Remark~\ref{topsectionremark}) and hence that $\mathbb{R}f_*\mmu_n \cong \mathbb{R}\pi_*\mmu_n$. 
\end{remark}

Returning to the case of surfaces, we have the following results on the cohomology of $Y$, which are well-known if $k=\Cee$:

\begin{proposition}\label{topsection} Let $f\colon Y \to B$ be a genus one fibration  and let $N$ be a positive integer prime to $\operatorname{char} k$. (Note: we do not assume that all fibers of $\pi$ are irreducible.) Suppose that $f$  has a singular fiber. Then:
\begin{enumerate}
\item[\rm(i)] The induced homomorphisms $H^1(B; \mmu_N) \to H^1(Y; \mmu_N)$ and $\pi_1(Y, *) \to \pi_1(B, *)$ are isomorphisms.
\item[\rm(ii)] The group $H^2(Y; \mmu_{N})$ is a free $\Zee/N\Zee$-module.
\item[\rm(iii)] There exists  a class $\theta \in H^2(Y; \mmu_{N})$ such that $\theta \cdot [F] =1$, i.e.\ the image of $\theta$ in
$H^0(B; R^2f_*\mmu_{N})\cong \Zee/N\Zee$ is $1$. 
\end{enumerate} 
\end{proposition}
\begin{proof} (i) Let $\nu \colon Y'\to Y$ be a finite \'etale cover, and take the Stein factorization of the composite morphism $Y'\to Y \to B$:
$$\begin{CD}
Y'@>{\nu}>> Y\\
@V{f'}VV @VV{f}V\\
B' @>>> B.
\end{CD}$$
Since $f$ has no multiple fibers, $B'\to B$ is \'etale. Consider the induced morphism $Y' \to Y\times_BB'$. Since $Y'\to Y$ and $Y\times_BB' \to Y$ are \'etale, $Y'\to Y\times_BB'$ is \'etale as well, and the statement of (i) is equivalent to the assertion that $Y'\to Y\times_BB'$ is an isomorphism. Replacing $Y$ by $Y\times_BB'$, which also has a singular fiber, it clearly suffices to prove that, if $\nu \colon Y'\to Y$ is \'etale and the general fiber of the composite morphism $f' = f\circ \nu \colon Y' \to B$ is connected, then $\nu$ is an isomorphism.  Note that, if $F$ is a smooth fiber of $f$, then $\nu^{-1}(F) = F'$ is a connected \'etale cover of $F$, of degree $\ell$, say. Thus $Y'$ is again a genus one fibration over $B$, with no multiple fibers and with Jacobian surface $J(Y') = X'$. Since all fibers of $f'$ are connected, the only possible singular fibers of $f\colon Y \to B$ are of type $I_k$, and at least one such fiber exists. Moreover, the fibers of $f'$ are smooth, if they lie over smooth fibers of $f$, or of type $I_{\ell k}$, if they lie over fibers of $f$ of type $I_k$. In particular, if $\chi(Y; \scrO_Y) = d>0$, so that the Euler number $e(Y) = c_2(Y) = 12d$, then $e(Y') = \ell e(Y) = 12d\ell$ and hence $\chi(Y'; \scrO_{Y'}) =  d\ell$.

On the other hand, $\chi(Y'; \scrO_{Y'}) = 1 - q(Y') + p_g(Y')$. Since  $Y'$ has no multiple fibers, $R^1(f')_*\scrO_{Y'} = \mathcal{L}$ is a line bundle on $B$ with $\deg \mathcal{L} = \deg (\omega_{X'/B})^{-1} < 0$, since $Y'$ and hence $X'$ have a nonmultiple singular fiber. Thus $H^0(B;R^1(f')_*\scrO_{Y'}) =0$ and it follows from the Leray spectral sequence that $q(Y') = h^1(Y'; \scrO_{Y'}) = g(B) =q(Y)$. Since $\nu \colon Y' \to  Y$ is \'etale,  $K_{Y'} = \nu^*K_Y $ and hence $p_g(Y') = h^0(Y'; K_{Y'}) = h^0(Y';\nu^*K_Y)= h^0(Y; \nu_*\nu^*K_Y)$. The homomorphism $H^0(Y; K_Y) \to  H^0(Y; \nu_*\nu^*K_Y)$ is an isomorphism, because  the restriction of the quotient $\nu_*\nu^*K_Y/K_K$ to a smooth fiber $F$ is of the form $\bigoplus_{i=1}^{\ell -1}\lambda^{-i}$, where $\lambda$ is a torsion line bundle of order $\ell$ and hence has no nonzero sections on $F$. Thus $p_g(Y') = h^0(Y; K_Y) = p_g(Y)$. But then $\chi(Y'; \scrO_{Y'}) = 1 - q(Y') + p_g(Y') = \chi(Y; \scrO_Y) = d$, contradicting the fact that $\chi(Y'; \scrO_{Y'}) =  d\ell$ and that $d>0$.

\medskip
\noindent (ii) Using (i), $H^1(Y; \mmu_N) \cong H^1(B; \mmu_N)$ is a free $\Zee/N\Zee$-module, and hence $H^3(Y; \mmu_N)$ is a free $\Zee/N\Zee$-module by Poincar\'e duality. The \'etale cohomology groups $H^i(Y; \mmu_N)$ are the cohomology of a finite complex $\mathcal{C}^\bullet$ of finite free $\Zee/N\Zee$-modules (see for example  \cite[p.\ 95, Theorem 4.9]{SGA4.5}). Thus, all but possibly one of the  $\Zee/N\Zee$-modules $H^i(\mathcal{C}^\bullet)$ is free. It is then a standard fact in homological algebra that $H^i(\mathcal{C}^\bullet)$ is a free $\Zee/N\Zee$-module for all $i$. In particular $H^2(Y; \mmu_N)$ is free.

\medskip
\noindent (iii) Using (ii), $H^2(Y; \mmu_{N})$ is a free $\Zee/N\Zee$-module.  By Poincar\'e duality, it is enough to show that the class $[F]$ of $F$ in $H^2(Y; \mmu_{N})$ is primitive, i.e.\ is not divisible by an integer $\ell > 1$ dividing $N$, which we may assume prime. Arguing by contradiction, suppose that $[F]=\ell\gamma$ for some prime $\ell$ and some $\gamma \in H^2(Y; \mmu_{N})$. Fix a point $p_0\in B$ and let $F_0$ be the corresponding fiber, which we assume is smooth. The image of $[F_0]$ in $H^2(Y; \mmu_{\ell})$ is zero, and hence, via the Kummer sequence, $\scrO_Y(F_0) = L^{\otimes \ell}$ for some line bundle $L$ on $Y$. Then $L$ restricts to an $\ell$-torsion line bundle on every fiber of $f$ and $L$ has degree zero on every component of every fiber. If $L|F$ is trivial on one smooth fiber $F$, or equivalently on all smooth fibers, then by semicontinuity $L|F$ has a section for  every fiber $F$, smooth or not. Since there are no multiple fibers, a standard argument (``Ramanujam's lemma," see e.g.\ \cite[II.12.2]{BHPV} or \cite[Exercise 1 p.\ 191]{Fr3}) shows that $L|F=\scrO_F$ for all fibers $F$ of $f$, and hence $L =f^*\lambda$ for some line bundle $\lambda$ on $B$. But this is clearly impossible, since then we would have $f^*\lambda^{\otimes \ell} = f^*\scrO_B(p_0)$. Applying $f_*$ then gives  $\lambda^{\otimes \ell} = \scrO_B(p_0)$ and hence $\ell \cdot \deg \lambda =1$, which is absurd.

Thus the restriction of $L$ to every smooth fiber has order $\ell$. Let $\nu \colon Y'\to Y$ be the $\mu_{\ell}$ cover defined by the $\ell^{\text{th}}$ root $L$ of $\scrO_Y(F_0)$ and the natural section of $\scrO_Y(F_0)$. By construction, $\nu_*\scrO_{Y'} = \scrO_Y \oplus L^{-1} \oplus \cdots \oplus L^{-(\ell-1)}$. Then $Y'$ is a smooth surface and the induced morphism $f'\colon Y'\to B$ has generic fiber equal to a nontrivial \'etale cover of a genus one curve. Hence $Y'$ is again a genus one fibration over $B$ with Jacobian surface $J(Y') = X'$. We are then in the situation of the proof of (i), except that, as  $(f')^*(p_0) = \nu^*F_0 = \ell G$ for some curve $G$ on $Y'$ isomorphic to $F_0$, $Y'$ has the unique multiple fiber $G$ over $p_0$ of multiplicity $\ell$,  and $G$ is tame. As before, if $\chi(Y; \scrO_Y) = d$, so that the Euler number $e(Y) = c_2(Y) = 12d>0$, then $e(Y') = \ell e(Y) = 12d\ell$ and hence $\chi(Y'; \scrO_{Y'}) =  d\ell$.

On the other hand, $\chi(Y'; \scrO_{Y'}) = 1 - q(Y') + p_g(Y')$. Since $G$ is tame, $R^1(f')_*\scrO_{Y'} = \mathcal{L}$ is a line bundle on $B$ with $\deg \mathcal{L} = \deg (\omega_{X'/B})^{-1} < 0$, and again $q(Y') = h^1(Y'; \scrO_{Y'}) = g(B) =q(Y)$. Since $\nu \colon Y' \to  Y$ is \'etale away from $F_0$, and is totally ramified at $F_0$, $K_{Y'} = \nu^*K_Y + (\ell-1)G$. Since $\nu^*K_Y|G$ is trivial and $\scrO_{Y'}(G)|G$ is torsion of order $\ell$, it is easy to check that $p_g(Y') = h^0(Y'; K_{Y'}) = h^0(Y';\nu^*K_Y)$, and, again arguing as in the proof of (i), $p_g(Y') = h^0(Y; K_Y) = p_g(Y)$. But then $\chi(Y'; \scrO_{Y'}) = 1 - q(Y') + p_g(Y') = \chi(Y; \scrO_Y) = d$, contradicting the fact that $\chi(Y'; \scrO_{Y'}) =  d\ell$ and that $d>0$. It follows that $[F]$ is not divisible in $H^2(Y; \mmu_{N})$, completing the proof of (iii).
\end{proof}

\begin{remark}\label{topsectionremark} 
If in addition  all fibers of $f$ are irreducible, then, by the isomorphism $R^1f_*\mmu_n \cong R^1\pi_*\mmu_n$ and an analysis of the Leray spectral sequences for $Y$ and for $X$, one can further show that  $H^2(B; R^1f_*\mmu_n) =0$. Using the Leray spectral sequence again, this gives an easier proof of (iii) under the additional assumption that all fibers of $f$ are irreducible.
\end{remark}

Pontrjagin square gives a quadratic form $\wp\colon H^2(X; \mmu_n) \to \Zee/2n\Zee$, which induces a similar quadratic form on $H^1(B; R^1\pi_*\mmu _n)$ (also denoted by $\wp$). In both cases it can be computed by lifting classes modulo $2n$ and squaring using the cupproduct. It is straightforward to check that the isomorphisms of Lemma~\ref{identify} are compatible with Pontrjagin square. We then have the following formula, special cases of which were established by Artin and Swinnerton-Dyer \cite{ASD}:

\begin{lemma}\label{ASDtoplemma} Suppose that $f$ or equivalently $\pi$ has a singular fiber. If $\alpha \in H^1(B; R^1\pi_*\mmu _n)$  and $Y$ is the corresponding genus one fibration, with $n$-section $D$, then
$$D^2 \equiv -n^2d + \wp(\alpha) \bmod 2n.$$
\end{lemma}
\begin{proof} To compute $\wp(\alpha)$, it suffices to lift $\alpha$ to a class in $H^1(B; R^1\pi_*\mmu _{2n})$ and compute its square. By Lemma~\ref{identify}, we can work with $[D] \in H^1(B; R^1f_*\mmu_n)$. The class $[D] \in H^2(Y; \mmu_{2n})$ does not induce a class in $H^1(B; R^1f_*\mmu _{2n})$ since its image in $H^0(B; R^2f_*\mmu_{2n})$ is nonzero. To correct this, applying Proposition~\ref{topsection}  with $N=2n$, there exists  $\theta\in H^2(Y; \mmu_{2n})$ such that $\theta \cdot F=1$.  Thus the class $[D] -n\theta\in H^2(Y; \mmu_{2n})$ maps to $0$ in  $H^0(B; R^2f_*\mmu_{2n})$ and so induces a class in $H^1(B; R^1f_*\mmu _{2n})$ which is clearly a lift of $[D]$. Computing, we find that 
\begin{align*}\wp(\alpha) &\equiv (D-n\sigma)^2 = D^2 -2n(D\cdot \theta) + n^2\theta^2 \bmod{2n}\\
 &\equiv D^2+n^2d \bmod{2n},
 \end{align*}
 since by the Wu formula \cite[Proposition 2.1]{Urabe}, $\theta^2 + K_Y\cdot \theta \equiv 0 \bmod 2$. This establishes the formula.
\end{proof}

From the formula 
$$D^2 = -2\deg R^1f_*\scrO_Y(-D) - (n+2)d,$$
we see that the $n$ possible values of $c_1(R^1f_*\scrO_Y(-D))$ mod $n$, or equivalently of $c_1(f_*\scrO_Y(D))$  mod $n$, determine and are determined by $D^2$ or by $\wp(\alpha)$ mod $2n$. Note that $D^2 + K_Y\cdot D = D^2 +n(d-2)\equiv 0 \bmod 2$, i.e.\ $D^2\equiv nd \bmod 2$. This is consistent with the lemma above since $-n^2d\equiv nd\bmod 2$ and  $\wp(\alpha)\equiv 0 \bmod 2$.

Before we state the main the the main theorem of this section, we note the following terminology: If $A$ is a free $\Zee/n\Zee$-module, for example $A=H^1(\Pee^1; R^1\pi_*\mmu _n)$, then  a class $\alpha \in A$ is \textsl{primitive} if the following holds: if $n'$ is a positive integer dividing $n$ and $\alpha = n'\alpha'$ for some $\alpha' \in A$, then $n'=1$. Similarly, if $\Lambda$ is a free $\Zee$-module, a class $\lambda \in \Lambda$ is \textsl{primitive} if the following holds: if $\lambda = a\lambda'$ for some positive integer $a$ and $\lambda'\in \Lambda$, then $a=1$.

\begin{theorem}\label{irreducible} Suppose that  $k= \Cee$.  Let $i\in \Zee/2n\Zee$ satisfy $i \equiv 0 \bmod 2$, and let $\mathcal{S}^0_{d,n,i}$ be the coarse moduli space of triples $(Y,f, D)$, where $f\colon Y \to \Pee^1$ is a genus one fibration with  Jacobian surface $\pi\colon X \to \Pee^1$, such that 
\begin{enumerate}
\item[\rm (i)] $\chi(Y;\scrO_Y) = \chi(X;\scrO_X) = d$;
\item[\rm (ii)] All fibers of $f$ or equivalently $\pi$ are irreducible;
\item[\rm (iii)]  The class $\alpha \in H^1(\Pee^1; R^1\pi_*\mmu _n)$ corresponding to $Y$ is primitive and $\wp(\alpha) = i$.
\end{enumerate}

Then $\mathcal{S}^0_{d,n,i}$ is irreducible.
\end{theorem}
\noindent \emph{Proof.}   We may assume that $d\geq 2$. Let $\mathcal{M}^0_d\subseteq \mathcal{M}_d$ be the coarse moduli space of elliptic surfaces $\pi \colon X \to \Pee^1$ with $\chi(X;\scrO_X) = d$ such that all fibers of $\pi$ are irreducible. As we have seen in the discussion of Theorem~\ref{deform}, $\mathcal{M}^0_d$ is irreducible. If $t\in \mathcal{M}^0_d$, we denote the corresponding elliptic surface by $X_t$. Let $\mathcal{S}^0_{d,n}$ be the coarse moduli space of triples $(Y,f, D)$, where $f\colon Y \to \Pee^1$ is a genus one fibration with  Jacobian surface $\pi\colon X \to \Pee^1$ satisfying (i) and (ii) in the statement of the theorem. There is a morphism $\mathcal{S}^0_{d,n} \to \mathcal{M}^0_d$ whose fiber over a  general point $t$ is $H^1(\Pee^1; R^1\pi_*\mmu _n) \cong [F]^\perp/(\Zee/n\Zee)\cdot [F]$, viewed as a subquotient of $H^2(X_t; \mmu_n)$. (In particular, if $X_t$ is general, then the pair $(X_t, \sigma)$ has no nontrivial automorphisms.) The irreducible  components of $\mathcal{S}^0_{d,n}$ then correspond to the orbits of the action of $\Gamma = \pi_1(\mathcal{M}^0_d, t)$ on $[F]^\perp/(\Zee/n\Zee)\cdot [F]$, which we can also identify with $\{\sigma, F\}^\perp$. The action of $\Gamma$ preserves both the divisibility and the Pontrjagin square of a class $\alpha \in \{\sigma, F\}^\perp$, so it suffices to show that $\Gamma$ acts transitively on the set of primitive classes $\alpha$ with $\wp(\alpha) =i$.

In $H^2(X_t; \Zee)$, the classes $\sigma$ and $F$ span a unimodular lattice, and 
$$\{\sigma, F\}^\perp = \Lambda \cong (2d-2)U \oplus d(-E_8),$$
 where $U$ is the rank two hyperbolic lattice and $-E_8$ is the negative of the root lattice $E_8$. The sublattice $\{\sigma, F\}^\perp$ of $H^2(X_t; \mmu_n)$ is the mod $n$ reduction of $\Lambda$, and if $\alpha$ is the reduction of an element $\lambda \in \Lambda$, then $\wp(\alpha) = \lambda^2 \bmod {2n}$. The action of $\Gamma$ on $\{\sigma, F\}^\perp\subseteq H^2(X_t; \mmu_n)$ is the  reduction of the action of $\Gamma$ on $\Lambda$. We begin by identifying this action, thanks to a result of 
 L\"onne \cite{Lonne} based on work of Ebeling \cite{Ebeling1, Ebeling2}:
 
 \begin{theorem} The image of $\Gamma$ in the orthogonal group $O(\Lambda)$ contains the index two subgroup $O^*(\Lambda)$ of elements of real spinor norm one.
 \end{theorem}
 \begin{proof} Begin with the degenerate Weierstrass equation $y^2z = 4x^3 -g_2xz^2 -g_3z^3$ defining a singular surface $X_0 \in \Pee (\scrO_{\Pee^1}(2d)\oplus \scrO_{\Pee^1}(3d)\oplus\scrO_{\Pee^1})$, where $g_2=0\in H^0(\Pee^1; \scrO_{\Pee^1}(4d))$ and $g_3\in H^0(\Pee^1; \scrO_{\Pee^1}(6d))$ is a section vanishing to order $6d-1$ at a point $p\in \Pee^1$ and hence vanishes simply at one other point. In local coordinates, $X_0$ has a singularity analytically of the form $y^2 = 4x^3 + z^{6d-1}$, which is of type $E_{12d-4}$, in Arnold's notation. According to Table 3 in \cite{Ebeling1}, the intersection pairing on the Milnor lattice of $E_{12d-4}$ is unimodular of rank $12d-4$ (as we shall also see directly). If  $X$ is a general regular elliptic surface with $\chi(\scrO_{X}) = d$, one can identify the Milnor lattice of $E_{12d-4}$ with a sublattice of $H^2(X; \Zee)$, which is orthogonal to $\sigma$ and $F$, hence is contained in $\Lambda$ and therefore equal to $\Lambda$ since both are unimodular.
 
The universal family of elliptic surfaces containing $X_0$  fails to be a versal deformation of the $E_{12d-4}$ singularity. In fact, the $E_{12d-4}$ singularity has a $\Cee^*$-action, and it is easy to check directly that the universal family of elliptic surfaces gives a versal deformation for the negative weight direction, and hence is transverse to the $\mu =$ constant stratum. Using results of Pinkham \cite{Pinkham}, one can also identify the Milnor fiber with $X_t - \sigma -F_c$, where $X_t$ is a smooth elliptic surface with $\chi(\scrO_{X_t}) =d$, $\sigma$ is a section on $X_t$, and $F_c$ is a cusp fiber. Thus the intersection pairing on the Milnor fiber is that of $H_2(X_t - \sigma -F_c; \Zee)$ and hence is isomorphic to $\Lambda$.

Without appealing to the theory of deformations of singularities with $\Cee^*$-action, one can argue directly (following \cite{Lonne}) as follows: By adding a small term to $g_2$ which does not vanish at $p$ and a small linear term to $g_3$ (in local coordinates), we obtain a new elliptic surface $X_s$ with local equation 
 $$y^2 = 4x^3 + a_1x + t^{6d-1}+a_2t+s,$$
 where we view $a_1$ and $a_2$ as fixed and $s$ as a parameter. The surface $X_s$ will have a singular point $(\xi, \eta, \tau)\in X_s$ exactly when when the partial derivatives $12x^2 + a_1$, $2y$, $(6d-1)t^{6d-2} + a_2$ all vanish for $(x,y,t) = (\xi, \eta, \tau)$.
If $\xi$ is a root of $12x^2 + a_1$, $\tau$ is a root of $(6d-1)t^{6d-2} + a_2$ and 
$$s = -(4\xi^3 + a_1\xi + \tau^{6d-1}+a_2\tau)= -\frac{2a_1}{3}\xi -\left(\frac{6d-2}{6d-1}\right)a_2\tau,$$ then the surface $X_s$ has an ordinary double point corresponding to $(x,y,t) = (\xi, 0, \tau)$.  One checks that, if $a_1$ and $a_2$ are general, then in this way we produce $12d-4$ different surfaces $X_s$, each with a single ordinary double point near the original $E_{12d-4}$ singularity. For each such surface and ordinary double point, let $\delta$ be the corresponding vanishing cycle and let $\Delta \subseteq \Lambda$ be the set of all such vanishing cycles. Then   $\Delta$ is the set of vanishing cycles for the $E_{12d-4}$ singularity. As such, $\Delta$ spans $\Lambda$ and  the Dynkin diagram corresponding to $\Delta$ is connected. Hence, if $\Gamma_{\Delta}$ is the group generated by reflections about the vanishing cycles in $\Delta$, then $\Delta$ is contained in a single $\Gamma_{\Delta}$-orbit. Moreover, the Dynkin diagram for $\Delta$ contains a certain subdiagram with $6$ vertices, which makes the pair $(\Lambda, \Delta)$ a complete vanishing lattice in the sense of \cite{Ebeling2}. It then follows by \cite[Theorem 5.3.4]{Ebeling2} that $\Gamma_{\Delta} = O^*(\Lambda)$.
 
On the other hand, it is easy to see that the deformations of $X_s$ are versal for the unique double point on $X_s$. The monodromy associated to this deformation acts on $\Lambda$ as the reflection in the corresponding vanishing cycle. Hence the image of the monodromy group contains $\Gamma_{\Delta}$ and thus it contains $O^*(\Lambda)$. 
  \end{proof}
  
  \begin{remark} Although the above proof analyzes the negative weight deformations of the $E_{12d-4}$ singularity $y^2 = 4x^3 + z^{6d-1}$, it is in many ways more natural to consider instead the singularity $y^2 = 4x^3 - c_2xz^{4d} - c_3z^{6d}$, whose Milnor fiber is diffeomorphic to $X_t - \sigma -F$, where $X_t$ is a smooth elliptic surface with $\chi(\scrO_{X_t}) =d$, $\sigma$ is a section,  and $F$ is a smooth fiber on $X_t$ isomorphic to the elliptic curve with equation $y^2 = 4x^3 - c_2x - c_3$.
  \end{remark}
 
 Next we have the following standard result,  due to Wall \cite{Wall} in the unimodular case:
 
 \begin{lemma}\label{Walllemma} Fix an even integer $j$. Then the group $O^*(\Lambda)$ acts transitively on the set of primitive classes $\lambda \in \Lambda$ with $\lambda^2 = j$. \qed
 \end{lemma}
 \begin{proof} By \cite{Wall}, the result holds with $O(\Lambda)$ instead of $O^*(\Lambda)$. In particular, fixing a hyperbolic summand $U$ of 
 $\Lambda$ with standard basis $\varepsilon, \delta$ ($\varepsilon^2 =  \delta^2=0$ and $\varepsilon\cdot \delta=1$), every primitive $\lambda \in \Lambda$ with $\lambda^2 = j$ is equivalent under $A\in O(\Lambda)$ to $\varepsilon+(j/2) \delta$. If $A\notin O^*(\Lambda)$, we can find a hyperbolic summand $U'$ of  $\Lambda$ orthogonal to $U$ and modify $A$ by a reflection about an element of $U'$ to adjust the spinor norm to be $1$.
 \end{proof}
 
 To complete the proof of Theorem~\ref{irreducible}, it therefore suffices to prove:
 
 \begin{lemma} For every $i\in \Zee/2n\Zee$ with $i\equiv 0\bmod 2$, the group $O^*(\Lambda)$ acts transitively on the set of primitive classes $\alpha \in \Lambda/n\Lambda$ with $\wp(\alpha) =i$.
 \end{lemma}
 \begin{proof} Fix an integer $j$ whose reduction mod $2n$ is $i$. By Lemma~\ref{Walllemma}, it suffices to show  that, if $\alpha$ is a primitive class in $\Lambda/n\Lambda$ with $\wp(\alpha) = i$, then there exists an integral lift $\tilde \alpha\in \Lambda$ such that $\tilde \alpha$ is primitive and $(\tilde \alpha)^2 = j$.  
 
 First we claim that we can assume that $\tilde \alpha$ is primitive. Begin with any lift $\tilde \alpha$ of $\alpha$ to $\Lambda$.  If $\tilde \alpha$ is not primitive, let $\tilde \alpha = a \lambda'$, where $a>1$ and $\lambda'$ is primitive. Since  $\alpha$ is primitive, $a$ and $n$ are relatively prime. We can assume by Lemma~\ref{Walllemma} that $\lambda' $ is a primitive vector in a standard hyperbolic summand $U$ of $\Lambda$. Choosing a hyperbolic summand $U'$ of $\Lambda$  orthogonal to $U$ and a primitive vector $\varepsilon '\in U'$ with $(\varepsilon ')^2 =0$, the vector $\tilde \alpha + n\varepsilon '$ is then a primitive vector in $\Lambda$ lifting $\alpha$.
 
Let $(\tilde \alpha)^2 = \ell = j+2nk$. Since $\tilde \alpha$ is primitive, we may as well assume, again by Lemma~\ref{Walllemma}, that $\tilde \alpha = (\ell/2)\varepsilon + \delta$, where $\varepsilon, \delta$ are a standard basis of a hyperbolic summand $U$ of $\Lambda$. But then $[(\ell/2)-nk]\varepsilon + \delta$ is another lift of $\alpha$ to $\Lambda$, and its square is $j$. This completes the proof of the lemma and hence of of Theorem~\ref{irreducible}.
 \end{proof}
 
\section{Existence of rigid bundles}

Throughout this section we consider only the case of elliptic fibrations with base $\Pee^1$.

\begin{definition} A rank $n$ vector bundle $V$ on $\Pee^1$ is \textsl{rigid} if there exists an integer $a$ such that $V \cong \scrO_{\Pee^1}(a)^k \oplus \scrO_{\Pee^1}(a-1)^{n-k}$ for some integer $k> 0$. In other words, up to a twist by $\scrO_{\Pee^1}(a)$, $V\cong \scrO_{\Pee^1}^k \oplus \scrO_{\Pee^1}(-1)^{n-k}$ with  $k>0$.
\end{definition}

\begin{lemma}\label{5.2} Let $V$ be a rank $n$ vector bundle   on $\Pee^1$ with $V =\bigoplus_{i=1}^n \scrO_{\Pee^1}(a_i)$ with $0=a_0\geq a_1\geq \cdots \geq a_n$, or equivalently such that $h^0(\Pee^1; V) \neq 0$ but $h^0(\Pee^1; V(-1)) = 0$. Then $V$ is rigid $\iff$ $a_n \geq -1$ $\iff$ $h^0(V) = \chi(V)$ $\iff$ $h^1(V) = 0$. \qed
\end{lemma}

In terms of divisors on genus one fibrations, we have:

\begin{lemma}\label{5.3} Let $Y\to \Pee^1$ be  a genus one fibration and let $D$ be an $f$-nef divisor on $Y$ such that $D$ is effective but $D-F$ is not effective. Then $f_*\scrO_Y(D)$ is rigid $\iff$  $\pi_*\scrO_X(D) = \scrO_{\Pee^1}^k \oplus \scrO_{\Pee^1}(-1)^{n-k}$ $\iff$ $h^0(Y; \scrO_Y(D)) = \chi(Y; \scrO_Y(D))$.
\end{lemma}
\begin{proof} Immediate from Lemma~\ref{5.2}, the Leray spectral sequence, and the fact that $R^1f_*\scrO_Y(D) =0$. 
\end{proof}

Of course, $f_*\scrO_Y(D)$ is rigid $\iff$  $R^1f_*\scrO_Y(-D)$ is rigid. In this section, it will be slightly simpler to work with $f_*\scrO_Y(D)$.

The goal of this section is to show the following:

\begin{theorem}\label{existence} Assume that $k=\Cee$. Let $d\geq 2$ and let $\mathcal{M}_d^0$ be the coarse moduli space of elliptic surfaces $\pi \colon X\to\Pee^1$ with $\chi(\scrO_X) = d $ and such that all fibers of $\pi$ are irreducible. Fix $n \geq 2$. Then there exists an open dense subspace of $\mathcal{M}_d^0$ consisting of elliptic surfaces $X$ such that, for every genus one fibration $f\colon Y\to \Pee^1$ whose Jacobian surface is $X$ and for every $f$-nef divisor $D$ on $Y$ with $D\cdot F = n$, the rank $n$ vector bundle $f_*\scrO_Y(D)$ is rigid.
\end{theorem}

The proof will be by induction on $d$.
In the case $d=2$, we have:

\begin{lemma} Let $X$ be an elliptic $K3$ surface, let $f\colon Y\to \Pee^1$ be a genus one fibration whose Jacobian surface is $X$, and let $D$ be an $f$-nef divisor  on $Y$ with $D\cdot F = n$. Then $f_*\scrO_Y(D)$ is rigid.
\end{lemma}
\begin{proof} This follows easily from Remark~\ref{remark3.9}. It is also easy to give a direct proof in case $D$ is minimal. In this case, there exists an irreducible curve $C$ in $|D|$. From the exact sequence $0\to \scrO_Y(-C) \to \scrO_Y \to \scrO_C \to 0$, it follows that $h^1(Y; \scrO_Y(-C)) =0$ and hence that $h^1(Y; \scrO_Y(C) )= 0$. Thus $h^0(Y; \scrO_Y(D)) = \chi(Y; \scrO_Y(D))$ and so $f_*\scrO_Y(D)$ is rigid by Lemma~\ref{5.3}.
\end{proof}

In the general case, the strategy is as follows. Using the discussion of the irreducible components from the previous section, it is enough to construct, for every integer $k\bmod n$, a single genus one fibration $f\colon Y \to \Pee^1$ with $\chi(\scrO_Y) = d$, with all fibers of $f$ irreducible, and a  divisor $D$ on $Y$ of relative degree $n$ such that $f_*\scrO_Y(D)$ is rigid and $c_1(f_*\scrO_Y(D))\equiv k\bmod n$. We will do so by finding suitable degenerations of genus one fibrations $Y$ to normal crossings fibrations $f\colon Y'\cup X \to \Pee^1\cup \Pee^1$, where $Y'$ is a genus one fibration with $\chi(\scrO_{Y'}) = d-1$ and $X$ is a rational elliptic surface, and $Y'$, $X$ are glued along a fiber. 

To check rigidity via degenerations, we shall use the following lemma:

\begin{lemma}\label{vectbundsdegen}  Let $\phi\colon \mathcal{B} \to \Delta$ be a smooth surface, fibered over a smooth curve $\Delta$, such that $\phi^{-1}(t) \cong \Pee^1$ for $t\neq t_0\in \Delta$ and such that $\phi^{-1}(t_0)$ has two components $C_1$ and $C_2$, each isomorphic to $\Pee^1$, glued at a point, so that $\phi^{-1}(t_0)$ has a single ordinary double point. Let $\mathcal{V}$ be a vector bundle on $\mathcal{B}$ such that $\mathcal{V}|C_1 = \scrO_{\Pee^1}^k\oplus \scrO_{\Pee^1}(-1)^{n-k}$, where $k>0$, and $\mathcal{V}|C_2 = \scrO_{\Pee^1}^n$. Then, there exists a nonempty Zariski open subset $U$ of $\Delta$ such that, for all $t\in U$, $\mathcal{V}|\phi^{-1}(t)$ is rigid, and in fact $\mathcal{V}|\phi^{-1}(t) \cong \scrO_{\Pee^1}^k\oplus \scrO_{\Pee^1}(-1)^{n-k}$.
\end{lemma}
\begin{proof} Let $\mathcal{V}_t = \mathcal{V}|\phi^{-1}(t)$. By Lemma~\ref{5.2}, it suffices to show that $h^0(\mathcal{V}_t) =k$ and $h^1(\mathcal{V}_t)= h^0(\mathcal{V}_t(-1))=0$ for $t\in U$. First we claim that  $h^1(\mathcal{V}_{t_0})=0$. This follows from the normalization exact sequence
$$0 \to \mathcal{V}_{t_0} \to n_*(\mathcal{V}|C_1 \oplus \mathcal{V}|C_2) \to (\Cee^n)_p\to 0,$$
where $p$ is the singular point of $\phi^{-1}({t_0})$ and $n \colon C_1 \amalg C_2 \to \phi^{-1}({t_0})$ is the normalization. Next, $h^0(\mathcal{V}_{t_0} \otimes \mathcal{L}) =0$, where $\mathcal{L}$ is a line bundle on $\mathcal{B}$ which restricts to $\scrO_{\Pee^1}(-1)$ on $\phi^{-1}(t)$, $t\in U$, restricts to $\scrO_{\Pee^1}(-1)$ on $C_1$, and is trivial on $C_2$, as again follows easily from the normalization exact sequence. So there is a nonempty Zariski open subset $U$ of $\Delta$ such that, for all $t\in U$, $h^1(\mathcal{V}_t)= h^0(\mathcal{V}_t(-1))=0$. Finally, $\chi(\mathcal{V}_t) = \chi(\mathcal{V}_{t_0}) =k> 0$, and thus $h^0(\mathcal{V}_t) =k$ as well.
\end{proof}

It is then enough to establish the following:

\begin{claim}\label{mainclaim} Fix an integer $k$, $0< k \leq n$ and an integer $d\geq 3$. Then there is a family $\mathcal{Y} \to \mathcal{B} \to \Delta$ where $\mathcal{Y}$ is a smooth threefold and $\mathcal{B}$ is a smooth surface, $\Delta$ is a smooth curve  as in Lemma~\ref{vectbundsdegen}, and a point $t_0\in \Delta$ with the  following properties. Denote by  $B_t$ the fiber of $\mathcal{B}$ over $t\in \Delta$, and similarly for $Y_t$, and let $f_t\colon Y_t \to B_t$ be the induced morphism. Then:
\begin{enumerate}
\item[\rm (i)] $B_t \cong \Pee^1$ for $t\neq t_0$ and $B_{t_0} \cong \Pee^1\cup\Pee^1$ meeting normally;
\item[\rm (ii)] The induced morphism $f_t \colon Y_t \to B_t$ gives $Y_t$ the structure of a genus one fibration with all fibers irreducible and $\chi(\scrO_{Y_t}) = d$, for $t\neq {t_0}$;
\item[\rm (iii)] $Y_{t_0} = Y'\cup X$, where $Y'$ is a genus one fibration with all fibers irreducible and $\chi(\scrO_{Y'}) = d-1$, $X$ is a rational elliptic surface with all fibers irreducible, $Y'$ and $X$ meet normally along a smooth fiber for both of the given fibrations, and the   morphism $f_{t_0} = f'\cup g\colon Y'\cup X \to \Pee^1\cup \Pee^1$ induces the two given fibrations on the two components of the normalization of $Y_{t_0}$;
\item[\rm (iv)] There exists a Cartier divisor $\mathcal{D}$ on $\mathcal{Y}$, such that $\mathcal{D}|Y' = D'$ satisfies: $(f')_*\scrO_{Y'}(D') = \scrO_{\Pee^1}^k\oplus \scrO_{\Pee^1}(-1)^{n-k}$ and  $\mathcal{D}|X = D$ satisfies: $g_*\scrO_X(D) = \scrO_{\Pee^1}^n$.
\end{enumerate}
\end{claim}

We prove the claim in several steps.
First, we can find the genus one fibration $f'\colon Y' \to \Pee^1$ and the divisor $D'$ by induction on $d$, beginning with the case $d=2$. 
Next we find appropriate divisors on rational elliptic surfaces:

\begin{theorem}\label{goodratsurf} Let $\pi \colon X\to \Pee^1$ be a rational elliptic surface and let $n\in \Zee$, $n\geq 1$. Then there exists a divisor $D$ on $X$ such that $D\cdot F = n$ and  $\pi_*\scrO_X(D) = \scrO_{\Pee^1}^n$. 
\end{theorem}

The proof follows from the next two lemmas:

\begin{lemma} 
Let $\pi \colon X\to \Pee^1$ be a rational elliptic surface and let $n\in \Zee$, $n\geq 1$. Let $D$ be a smooth rational curve in $X$ such that $D^2 = n-2$. Then $\pi_*\scrO_X(D) = \scrO_{\Pee^1}^n$. 
\end{lemma}
\begin{proof} It follows  from the exact sequence $$0\to \scrO_X \to \scrO_X(D) \to \scrO_{\Pee^1}(n-2) \to 0$$
 that $h^1(\scrO_X(D)) = 0$, and clearly $h^2(\scrO_X(D)) = h^2(\scrO_X(-D-F) ) =0$. From the exact sequence $$0\to \scrO_X(-D) \to \scrO_X \to \scrO_D \to 0,$$
  it follows that $h^2(\scrO_X(-D))=0$, and hence that $h^0(\scrO_X(D+K_X)) = h^0(\scrO_X(D-F) )= 0$. 
  
  By adjunction, $-2 = D^2+D\cdot K_X = n-2 +D\cdot K_X=n-2-D\cdot F$. Thus $D\cdot F = n$. It follows that $\chi(\scrO_X(D)) = h^0(\scrO_X(D)) = \frac12(D^2 -D\cdot K_X)+1 = \frac12(2n-2)+1=n$. 
  
  By Lemma~\ref{5.3},  $\pi_*\scrO_X(D) = \scrO_{\Pee^1}^k \oplus \scrO_{\Pee^1}(-1)^{n-k}$ is rigid, and $$k = h^0(\Pee^1; \pi_*\scrO_X(D)) = h^0(X;\scrO_X(D))=n.$$ Thus $\pi_*\scrO_X(D) = \scrO_{\Pee^1}^n$.
\end{proof}

\begin{lemma} 
Let $\pi \colon X\to \Pee^1$ be a rational elliptic surface and let $n\in \Zee$, $n\geq 1$. Then there exists a  smooth rational curve $D$ in $X$ such that $D^2 = n-2$.  
\end{lemma}
\begin{proof} If $n=1$, then we can just take a section of $X$, or equivalently an exceptional curve. Otherwise, there is a blowdown $\rho\colon X \to \bar{X}$, where $\bar{X}$ is a smooth minimal rational surface with a smooth anticanonical divisor, and hence $\bar{X}= \mathbb{F}_0$, $\mathbb{F}_2$, or $\Pee^2$. Thus $X$ dominates $\mathbb{F}_0$, $\mathbb{F}_1$, or $\mathbb{F}_2$. Moreover, $X$ is a blowup of either  $\mathbb{F}_0$, $\mathbb{F}_1$, or $\mathbb{F}_2$ at $8$ points, and if $X$ is a blowup of $\mathbb{F}_2$, then the first point of the blowup does not lie on the negative section. It then follows easily that $X$ either simultaneously dominates $\mathbb{F}_0$ and  $\mathbb{F}_1$, or $X$  simultaneously dominates $\mathbb{F}_1$ and  $\mathbb{F}_2$. It thus suffices to show: if $n$ is even and $n\geq 2$, then there exist base point free linear systems on $\mathbb{F}_0$ and  $\mathbb{F}_2$ whose general members are smooth rational curves $D$ with $D^2 =n-2$, and if $n$ is odd and $n\geq 3$, then there exists a  linear system on $\mathbb{F}_1$ whose general member is a smooth rational curve $D$ with $D^2 =n-2$. The proper transform in $X$ of a general $D$ will then have the desired properties.

If $n=2$, then we can take $D$ to be a general fiber of the ruling on $\mathbb{F}_a$, $a=0,1,2$. So we can assume $n\geq 3$. In case $n=2a$ is even, $a\geq 2$, and $X$ dominates $\mathbb{F}_2$, choose the base point free linear system $|\sigma + aF|$, where $\sigma$ is the negative section. Then the general member $D$ of $|\sigma + aF|$ is smooth and satisfies $D^2 = 2a-2 = n-2$. The case where $X$ dominates $\mathbb{F}_0$ is similar, using the linear system $|\sigma + (a-1)F|$, where $\sigma$, $F$ are the classes of the two rulings on $\mathbb{F}_0$. If $n= 2a+1$ is odd and $n\geq 3$, so that $a\geq 1$, choose the base point free linear system $|\sigma + aF|$ on $\mathbb{F}_1$, where $\sigma$ is the negative section. The general member $D$ of $|\sigma + aF|$ is smooth and satisfies $D^2 = 2a-1 = n-2$.
\end{proof}

Next we  construct a degenerating family of elliptic surfaces with a section which is the analogue of the family $\mathcal{Y}$ of genus one fibrations:

\begin{theorem}\label{existdegens} Let $\pi \colon X\to \Pee^1$ be a smooth rational elliptic surface, let $\pi' \colon X'\to \Pee^1$ be an elliptic surface with $\chi(\scrO_{X'}) = d-1$, let $F\subseteq X$ and $F' \subseteq X'$ be two smooth fibers which are isomorphic as elliptic curves. Then there exists a family $\mathcal{X} \to \mathcal{B} \to \Delta$ where $\mathcal{X}$ is a smooth threefold and $\mathcal{B}$ is a smooth surface, $\Delta$ is a smooth curve and $t_0\in \Delta$ as in Lemma~\ref{vectbundsdegen}, with the following properties: let $B_t$ be the fiber of $\mathcal{B}$ over $t\in \Delta$, and similarly for $X_t$, and let $\pi_t\colon X_t \to B_t$ be the induced morphism. Then: 
\begin{enumerate}
\item[\rm(i)] $B_t \cong \Pee^1$ for $t\neq t_0$ and $B_{t_0} \cong \Pee^1\cup\Pee^1$ meeting normally;
\item[\rm(ii)] The induced morphism $\pi_t \colon X_t \to B_t\cong \Pee^1$ realizes $X_t$ as an elliptic surface over $\Pee^1$ with  $\chi(\scrO_{X_t}) = d$, for $t\neq {t_0}$;
\item[\rm(iii)] $X_{t_0} = X'\cup X$, where $X'$ and $X$ meet normally  along the curves $F\cong F'$, and the   morphism $\pi_{t_0} = \pi'\cup \pi\colon X'\cup X \to \Pee^1\cup \Pee^1$ induces the two given fibrations on the two components of the normalization of $X_{t_0}$.
\end{enumerate}
\end{theorem}
\begin{proof} The surface $X'$ is defined by  sections $g_2'\in H^0(\Pee^1; \scrO_{\Pee^1}(4(d-1)))$ and $g_3'\in H^0(\Pee^1; \scrO_{\Pee^1}(6(d-1)))$. Let $p\in \Pee^1$ be the point corresponding to the fiber $F'$. Since $\scrO_{\Pee^1}(4d) \otimes \scrO_{\Pee^1}(-4p) \cong \scrO_{\Pee^1}(4(d-1))$ and $\scrO_{\Pee^1}(6d) \otimes \scrO_{\Pee^1}(-6p) \cong \scrO_{\Pee^1}(6(d-1))$,  $g_2'$  defines a section  $g_2\in H^0(\Pee^1; \scrO_{\Pee^1}(4d))$ vanishing to order $4$ at $p$, and likewise $g_3'$  defines a section  $g_3\in H^0(\Pee^1; \scrO_{\Pee^1}(6d))$ vanishing to order $6$ at $p$. The Weierstrass equation $y^2z = 4x^3 - g_2xz^2 - g_3z^3$   defines a hypersurface $\overline{X}$ in the $\Pee^2$-bundle $\Pee(\scrO_{\Pee^1}(2d) \oplus \scrO_{\Pee^1}(3d) \oplus \scrO_{\Pee^1})$, with a cuspidal fiber $\overline{E}$ over $p$, and $\overline{X}$ has a simple elliptic singularity at the cusp of $\overline{E}$. More precisely, fixing a local coordinate $u$ on $\Pee^1$ at $p$ and a local section of $\scrO_{\Pee^1}(1)$ at $p$, and hence of $\scrO_{\Pee^1}(4d)$ and $\scrO_{\Pee^1}(6d)$, the singularity defined by the above Weierstrass equation is analytically $y^2 = 4x^3 - c_2xu^4 -c_3u^6$. For simplicity, we shall just consider the case where $c_3\neq 0$, i.e.\ where the $j$-invariant of the fiber $F$ is not $1728$. 

Because the restriction homomorphism 
$$H^0(\Pee^1; \scrO_{\Pee^1}(4d)) \to \scrO_{\Pee^1}(4d)/u^5\scrO_{\Pee^1}(4d)$$ is surjective, for $i=0,1,2,3$ there exist sections $h_2^{(i)}\in H^0(\Pee^1; \scrO_{\Pee^1}(4d))$ such that, using the local coordinate $u$ at $p$ and the fixed trivialization of $\scrO_{\Pee^1}(4d))$ chosen above, $h_2^{(i)}(u) = u^i + O(u^5)$. Likewise, for $i=0, \dots, 5$, there exist sections $h_3^{(i)}\in H^0(\Pee^1; \scrO_{\Pee^1}(6d))$ such that, using the local coordinate $u$ at $p$ and the fixed trivialization of $\scrO_{\Pee^1}(6d))$ chosen above, $h_3^{(i)}(u) = u^i + O(u^7)$. Let $v$ be a coordinate of $\mathbb{A}^1$. For fixed $\alpha_1, \dots, \alpha_4, \beta_1, \dots, \beta_6\in k$, consider the family $\overline{\mathcal{X}} \subseteq \Pee(\scrO_{\Pee^1}(2d) \oplus \scrO_{\Pee^1}(3d) \oplus \scrO_{\Pee^1}) \times \mathbb{A}^1$ of elliptic surfaces parametrized by $v\in \mathbb{A}^1$ defined by $y^2z = 4x^3 - G_2(v)xz^2 - G_3(v)z^3$, where
\begin{align*}
G_2(v) &= g_2 + \sum_{i=1}^4\alpha_iv^ih_2^{(4-i)};\\
G_3(v) &= g_3 + \sum_{i=1}^6\beta_iv^ih_3^{(6-i)}.
\end{align*} 
Note that there is an induced morphism $\overline{\mathcal{X}} \to \Pee^1\times \mathbb{A}^1$ as well as a morphism $\overline{\mathcal{X}} \to  \mathbb{A}^1$ whose fiber over $0$ is $\overline{X}$. It is easy to check that, if the $\alpha_1, \dots,  \beta_6$ are not all $0$, then the general fiber is an elliptic surface over $\Pee^1$ with at worst rational double points. There is a point $q\in \overline{\mathcal{X}}$ lying over $(p,0)$ and corresponding to $x=y=0, z=1$; the local equation for $\overline{\mathcal{X}}$ at $q$ is 
$$y^2 = 4x^3 - (c_2u^4 + \sum_{i=1}^4\alpha_iv^iu^{4-i} + O(u^5))x - (c_3u^6 + \sum_{i=1}^6\beta_iv^iu^{6-i}+ O(u^7)).$$
Make a weighted blowup of the open subset $\{z\neq 0\}$ of $\Pee(\scrO_{\Pee^1}(2d) \oplus \scrO_{\Pee^1}(3d) \oplus \scrO_{\Pee^1}) \times \mathbb{A}^1$ with coordinates $u,v,x,y$, where $u$ and $v$ have weight $1$, $x$ has weight $2$ and $y$ has weight $3$. Thus the exceptional divisor is a weighted projective space $\Pee(1,1,2,3)$. A standard calculation shows that the proper transform $\widetilde{\mathcal{X}}$ of $\overline{\mathcal{X}}$ in this weighted blowup has the effect of resolving the simple elliptic singularity at $q\in \overline{X}$, and in fact is the minimal  resolution of the simple elliptic singularity, as well as resolving the cusp singularity on $\overline{E}$. If $\tilde{X}$ is the proper transform of $\overline{X}$, then the new exceptional divisor on $\tilde{X}$ is an elliptic curve $\tilde{F}$ of self-intersection $-1$, the proper transform of $\overline{E}$ is an exceptional curve $E$, and contracting $E$ gives a morphism  $\tilde{X}\to X'$ which is the blowup of $X'$ at a point on the fiber $F$. 

The exceptional divisor of the morphism $\widetilde{\mathcal{X}} \to \overline{\mathcal{X}}$  is a hypersurface $\hat{X}$ in the weighted projective space $\Pee(1,1,2,3)$ with homogeneous coordinates $u,v,x,y$. As $\hat{X}$ is defined by the homogeneous degree $6$ equation 
$$y^2 = 4x^3 - (c_2u^4 + \sum_{i=1}^4\alpha_iv^iu^{4-i} )x - (c_3u^6 + \sum_{i=1}^6\beta_iv^iu^{6-i}),$$ 
it is a degree one (generalized) del Pezzo surface in $\Pee(1,1,2,3)$, i.e.\ rational double points are allowed. Clearly, by choosing the $\alpha_i, \beta_j$, one can arrange that $\hat{X}$ is an arbitrary del Pezzo surface subject to the condition that  the curve defined by $u=0$ is isomorphic to $F$. If $\hat{X}$ is  smooth  (no rational double points), then $\mathcal{X}$ will be smooth as well, at least in a neighborhood of the fiber over $0$. By a standard construction in threefold birational geometry (a ``Type I modification," see for example \cite [p.\ 13]{Birat}), one can flip the exceptional curve $E$ on $\tilde{X}$ to $\hat{X}$. The new fiber over $0\in\mathbb{A}^1$ then consists of $X'$ together with the blowup $X$ of the del Pezzo surface $\hat{X}$ at the base point of $|-K_{\hat{X}}|$ ($x=1, y=2, u=v=0$), which is then a rational elliptic surface.  The construction can be summarized in the following picture:
$$\begin{matrix}
 \tilde{X}\cup \hat{X} & \subseteq & \widetilde{\mathcal{X}} & \dasharrow & \mathcal{X} & \supseteq & X'\cup X\\
 \downarrow & {} &\downarrow & {} &   \downarrow {} & {}& {} \\
 \overline{X} & \subseteq & \overline{\mathcal{X}} & {} & \mathcal{B} &{} &{}\\
& {} &\downarrow & \swarrow & & &\\
 & & \Pee^1\times \mathbb{A}^1 & & & 
\end{matrix}$$
 The birational morphism $\widetilde{\mathcal{X}}  \dasharrow  \mathcal{X}$ is the Type I modification, which contracts the exceptional curve $E$ on $\tilde{X}$ and blows up the corresponding point of $\hat{X}$ to obtain a rational elliptic surface $X$.  
The inverse image of $(p,0)\in \Pee^1\times \mathbb{A}^1$ in $\overline{\mathcal{X}}$ is the cuspidal fiber $\overline{E}$. Thus the preimage of $(p,0)$ in the weighted blowup of $\overline{\mathcal{X}}$ consists of the union of the exceptional divisor $\hat{X}$ and the proper transform of $\overline{E}$, namely  $E$. After flipping $E$, the inverse image of $(p,0)\in \Pee^1\times \mathbb{A}^1$ is just the divisor $X$, and hence the morphism $\mathcal{X}\to \Pee^1\times \mathbb{A}^1$ factors through the blowup $\mathcal{B}$ of $\Pee^1\times \mathbb{A}^1$ at the point $(p,0)$. Clearly the fiber over $0\in \mathbb{A}^1$ of the morphism $\mathcal{X}\to \mathbb{A}^1$ is $X'\cup X$, and for $t\neq 0$, $t$ in a nonempty Zariski open subset of $\mathbb{A}^1$, the fiber over $t$ of the morphism $\mathcal{X}\to \mathbb{A}^1$ is a smooth elliptic surface. It is easy to check that the induced morphism $X'\cup X \to \Pee^1\cup \Pee^1$ induces the given fibrations on $X'$ and $X$. Replace $\mathbb{A}^1$ by the nonempty open set $\Delta$ which is the complement of points other than $0$ where $\mathcal{X}\to \mathbb{A}^1$ fails to be smooth, and $\mathcal{X}, \mathcal{B}$ by the respective preimages of $\Delta$. It is then straightforward to see that $\mathcal{X}\to \mathcal{B}$ is as claimed.
\end{proof}

We now complete the proof of Claim~\ref{mainclaim} and thus of Theorem~\ref{existence}. Let $d\geq 3$. By induction there exists an elliptic surface $\pi'\colon X'\to \Pee^1$  with $\chi(\scrO_{X'}) = d-1 \geq 2$, with all fibers of $\pi'$ irreducible, and such that there exists a class $\xi' \in H^1(X'; R^1(\pi')_*\mmu_n)$ so that the corresponding pair $(Y', D')$ satisfies: $(f')_*\scrO_{Y'}(D') = \scrO_{\Pee^1}^k\oplus \scrO_{\Pee^1}(-1)^{n-k}$. Let $\pi\colon X\to \Pee^1$ be a rational elliptic surface with all fibers irreducible, let $D$ be a divisor of fiber degree $n$ on $X$ such that $\pi_*\scrO_X(D) = \scrO_{\Pee^1}^n$ (whose existence is guaranteed by Theorem~\ref{goodratsurf}), and let $\xi \in H^1(X; R^1\pi_*\mmu_n)$ be the class corresponding to the pair $(X, D)$. Using the section $\sigma \subseteq X$, we can identify $H^1(X; R^1\pi_*\mmu_n)$ with $\{\sigma, F\}^\perp \subseteq H^2(X; \mmu_n)$, and similarly for $X'$.  Let $\mathcal{X} \to \mathcal{B} \to \Delta$ be the family constructed in Theorem~\ref{existdegens}. From the Mayer-Vietoris sequence for $X_0 = X'\cup X$, namely
$$0\to (\mmu_n)_{X_0} \to (\mmu_n)_{X'}\oplus (\mmu_n)_{X}\to (\mmu_n)_{F}\to 0,$$
the pair $(\xi', \xi)$ induces an element in $H^2(X; \mmu_n) \oplus H^2(X'; \mmu_n)$ which is orthogonal to the classes of $F$ and $F'$, and hence an element of $\lambda_0\in H^2(X_0; \mmu_n)$ orthogonal to the class of a fiber in each surface.  

By the proper base change theorem, after replacing $\Delta$ by an \'etale cover, which we continue to denote by $\Delta$, there exists a class $\lambda \in H^2(\mathcal{X}; \mmu_n)$ whose restriction to $H^2(X_0; \mmu_n)$ is $\lambda_0$, and hence $\lambda \cdot [F] = 0$ for every fiber $F$ of the morphism $\mathcal{X} \to \mathcal{B}$. So finally there is an induced element of $H^1(\mathcal{X}; R^1\pi_*\mmu_n)$, also denoted by $\lambda$, which restricts in the appropriate sense on $X'$ and $X$ to give the classes $\xi'$ and $\xi$, respectively. 

By the remarks at the beginning of the last section, since all fibers of the morphism $\mathcal{X}\to \mathcal{B}$ are irreducible, the class $\lambda\in H^1(\mathcal{X}; R^1\pi_*\mmu_n)$ corresponds to a principal homogeneous space $\mathcal{Y}\to \mathcal{B}$, together with a divisor $\mathcal{D}$ on $\mathcal{Y}$ of relative degree $n$. By construction, it is clear that the conditions of Claim~\ref{mainclaim} are satisfied. This completes the proof of Claim~\ref{mainclaim} and  Theorem~\ref{existence}. \qed

\bigskip
\noindent
Department of Mathematics \\
Columbia University \\
New York, NY 10027 \\
USA

\bigskip
\noindent
{\tt dejong@math.columbia.edu, rf@math.columbia.edu}

\end{document}